\documentclass{siamltex}
\usepackage{amsmath}
\usepackage{amsfonts}
\usepackage{amssymb}
\usepackage{sw20siam}
\usepackage{mathrsfs}
\usepackage{graphicx}
\usepackage{mathtools}
\usepackage{algorithm,algorithmic}

\usepackage{enumitem}

\graphicspath{{fig/}}

\makeatletter
\numberwithin{equation}{section}
\numberwithin{figure}{section}

\newtheorem{thm}[theorem]{Theorem}
\newtheorem{lem}[theorem]{Lemma}
\newtheorem{prop}[theorem]{Proposition}
\newtheorem{cor}[theorem]{Corollary}
\newtheorem{rem}[theorem]{Remark}

\usepackage{color}
\definecolor{red}{rgb}{1,0,0} % red
\definecolor{green}{rgb}{0,1,0} % green
\definecolor{blue}{rgb}{0,0,1} % blue
\definecolor{darkblue}{rgb}{0,0,0.6}
\definecolor{darkred}{rgb}{0.6,0,0}

\usepackage[colorlinks,
            linkcolor=blue,
            anchorcolor=darkblue,
            citecolor=darkblue
           ]{hyperref}

\newcommand{\eps}{\varepsilon}  % greeks

\newcommand{\R}{\mathbb{R}}
\newcommand{\bS}{\mathbb{S}}
\newcommand{\me}{\mathrm{e}}

\newcommand{\Pb}{\mathbb{P}}
\newcommand{\E}{\mathbb{E}}
\newcommand{\F}{\Phi}
\newcommand{\tr}{\mathsf{T}}

\newcommand{\vA}{\varLambda}
\newcommand{\vN}{N}

\newcommand{\vpi}{\Pi}

\title{A Q-learning algorithm for discrete-time linear-quadratic control
with random parameters of unknown distribution: convergence and stabilization}

\author{Kai~Du\thanks{Shanghai Center for Mathematical Sciences, Fudan University, 
Shanghai 200438, China (email: {\tt kdu@fudan.edu.cn}). K.~Du was partially supported 
by National Key R{\&}D Program of China (No.~2018YFA0703900), 
by National Natural Science Foundation of China (No.~11801084), 
by Shanghai Program for Professor of Special Appointment (Eastern Scholar) 
at Shanghai Institutions of Higher Learning, 
and by Natural Science Foundation of Shanghai (No.~20ZR1403600).}
\and Qingxin~Meng\thanks{Department of Mathematics, Huzhou University, Zhejiang 313000, China 
(email: {\tt mqx@zjhu.edu.cn}).
Q.~ Meng was partially supported 
by the National Natural Science Foundation of China (No.~11871121), 
and by the Natural Science Foundation of Zhejiang Province (No.~LY21A010001).}
\and Fu~Zhang\thanks{Corresponding author. College of Science, University of Shanghai for Science and Technology,
Shanghai 200093, China (email: {\tt fuzhang82@gmail.com}).
F.~Zhang was partially supported 
by the National Natural Science Foundation of China (No. 11701369, 12071292).}}

\begin{document}

\maketitle
\begin{abstract}
This paper studies an infinite horizon optimal control problem for
discrete-time linear systems and quadratic criteria, both with random
parameters which are independent and identically distributed with
respect to time. A classical approach is to solve an algebraic Riccati
equation that involves mathematical expectations and requires certain
statistical information of the parameters. In this paper, we propose
an online iterative algorithm in the spirit of Q-learning for the
situation where only one random sample of parameters emerges at each time step. 
The first theorem proves the equivalence
of three properties: the convergence of the learning sequence, the
well-posedness of the control problem, and the solvability of the
algebraic Riccati equation. The second theorem shows that the adaptive
feedback control in terms of the learning sequence stabilizes the
system as long as the control problem is well-posed. Numerical examples
are presented to illustrate our results.
\end{abstract}
\keyphrases{linear quadratic optimal control, random parameters, Q-learning, convergence, stabilization.}
\AMclass{49N10, 93E35, 93D15.}

\section{Introduction and main results}

This paper aims to propose an online algorithm to solve the infinite
horizon linear quadratic (LQ) control problem for discrete-time systems
with random parameters. Given an initial state $x_{0}=x\in\R^{n}$,
the system evolves as
\begin{equation}
x_{t+1}=\vA_{t+1}\begin{bmatrix}x_{t}\\
u_{t}
\end{bmatrix},\quad t=0,1,2,\dots,\label{eq:system}
\end{equation}
where $x_{t}\in\R^{n}$ denotes the state and $u_{t}\in\R^{m}$ the
control at time $t$; the cost function is defined as
\begin{equation}
J(x,u_{\cdot})=\sum_{t=0}^{\infty}\big[x_{t}^{\tr},u_{t}^{\tr}\big]\vN_{t+1}\begin{bmatrix}x_{t}\\
u_{t}
\end{bmatrix}.\label{eq:cost}
\end{equation}
The random parameters $\vA_{t+1}$ and $N_{t+1}$ that affect the
system from $t$ to $t+1$ are not exposed until time $t+1$. The
objective of this control problem is to minimize the expected value
of the cost function among all admissible controls; the selection
of such a control $u_{t}$ is only based on the information of parameters
up to time $t$. In this paper, we assume that $N_{t}$ is positive
semidefinite, and the random matrices $[\vA_{t}^{\tr},\vN_{t}]$ with
$t=1,2,\dots$ are independent and identically distributed, but \emph{their
statistical information is previously unknown}. In what follows, $[\vA^{\tr},\vN]$
denotes an independent copy of $[\vA_{1}^{\tr},\vN_{1}]$.

The study of optimal control of discrete-time linear systems with
independent random parameters can date back to Kalman~\cite{kalman1961control}
in 1961, motivated by random sampling systems~\cite{Kalman1959405}.
Unsurprisingly, such models arise also in many other situations, for
instance, control systems that involve state and control-dependent
noise~\cite{drenick1964optimal,martin1975stability,Harris1978213},
digital control of diffusion processes~\cite{Tiedemann1984449}, and
macroeconomic systems~\cite{chow1975analysis,aoki1976optimal} where
the randomness of parameters of econometric models is taken into account.

Due to the wide application background, the LQ problem with random
parameters has been extensively studied (see~\cite{kalman1961control,drenick1964optimal,alspach1973dual,Athans1977491,Ku1977866,DeKoning1982443,Morozan198389,Yaz1988407,Pronzato1996855,Beghi19981031,huang2006infinite,Wang2016379}
for example). As far as the infinite horizon problem is concerned,
the key issues addressed mostly in the literature include:
\begin{itemize}[leftmargin=6ex]
\item to determine whether the LQ problem
is \emph{well-posed}, i.e., the value function
\begin{equation}
V(x):=\inf_{u_{\cdot}}\E[J(x,u_{\cdot})]
~~~~~~~~~~~~\label{eq:value}
\end{equation}
is finite for all $x\in\R^{n}$; and if it is well-posed,
\item to construct an optimal control $u_{\cdot}^{\star}$ for each $x$
such that $\E[J(x,u_{\cdot}^{\star})]=V(x)$.
\end{itemize}
Moreover, it is known that the well-posedness issue has an intimate
link to stabilizability of the system (\ref{eq:system}) which in
itself is an important topic and has also been widely discussed (see
\cite{martin1975stability,Ku1977866,DeKoning1982443,Morozan198389,yaz1985stabilization}
for example). 

For the above issues, a commonly used approach in the literature is
to apply stochastic dynamic programming~\cite{aoki1967optimal} to
the LQ problem, resulting in an algebraic Riccati equation (ARE) that
characterizes the value function and the optimal (feedback) control.
In this sense, the problem can be perfectly solved if the distribution
of the random parameters are known. To capture more mathematical insights,
let us quickly look at the informal derivation. The value function, if it is
finite, is believed to be a quadratic form, say, $V(x)=x^{\tr}Kx$
with some positive semidefinite matrix $K$, then Bellman's principle
of optimality gives that
\begin{equation}
\begin{aligned}x^{\tr}Kx & =\inf_{u_{t}}\E\bigg\{\big[x_{t}^{\tr},u_{t}^{\tr}\big]\vN_{t+1}\begin{bmatrix}x_{t}\\
u_{t}
\end{bmatrix}+x_{t+1}^{\tr}Kx_{t+1}\,\bigg|\,x_{t}=x\bigg\}\\
 & =\inf_{u_{t}}\E\bigg\{\big[x_{t}^{\tr},u_{t}^{\tr}\big](\vN_{t+1}+\vA_{t+1}^{\tr}K\vA_{t+1})\begin{bmatrix}x_{t}\\
u_{t}
\end{bmatrix}\bigg|\,x_{t}=x\bigg\}\\
 & =\min_{u\in\R^{m}}\big[x^{\tr},u^{\tr}\big]\big(\E[\vN_{t+1}+\vA_{t+1}^{\tr}K\vA_{t+1}]\big)\begin{bmatrix}x\\
u
\end{bmatrix}\quad\forall\,x\in\R^{n};
\end{aligned}
\label{eq:xKx}
\end{equation}
the last unconditional extremum can be solved out explicitly. To proceed
further, let us introduce some notations used frequently in this paper.
Let $\bS_{+}^{d}$ be the set of positive semidefinite $d\times d$-matrices
with $d=n+m$, and for any $P\in\bS_{+}^{d}$ we refer to its certain
submatrices according to the following partition:
\begin{equation}
P=\begin{bmatrix}P_{xx} & P_{xu}\\
P_{ux} & P_{uu}
\end{bmatrix}\quad\text{with }\ P_{xx}\in\R^{n\times n},\label{eq:Q-decomp-1}
\end{equation}
and define the following two mappings:
\begin{equation}
\begin{aligned}\vpi(P) & :=P_{xx}-P_{xu}P_{uu}^{+}P_{ux},\\
\Gamma(P) & :=-P_{uu}^{+}P_{ux},
\end{aligned}
\label{eq:pi}
\end{equation}
where $P_{uu}^{+}$ denotes the Moore--Penrose pseudoinverse of $P_{uu}$.
Using the above notations, one can easily obtain from (\ref{eq:xKx})
that
\begin{equation}
K=\vpi(\E[\vN+\vA^{\tr}K\vA]),\label{eq:riccati}
\end{equation}
which is exactly the algebraic Riccati equation (ARE) for LQ problem
(\ref{eq:system})--(\ref{eq:cost}). Moreover, the infimum in (\ref{eq:xKx})
can be achieved by taking 
\[
u_{t}=\Gamma(\E[\vN+\vA^{\tr}K\vA])x_{t},
\]
which gives the optimal feedback control of the LQ problem. In principle,
to compute the expectation in (\ref{eq:riccati}) one need know certain
statistical information of the parameters $N_{t}$ and $\vA_t$.

The above argument can be rigorized without much effort; actually,
under rather general settings, LQ problem (\ref{eq:system})--(\ref{eq:cost})
is well-posed if and only if ARE (\ref{eq:riccati}) has a solution
(see~\cite{DeKoning1982443,Morozan198389} or Theorem~\ref{thm:main}
below). Utilizing this relation, some useful criteria for well-posedness
of infinite horizon LQ problems are obtained in the papers~\cite{drenick1964optimal,Athans1977491,Ku1977866,DeKoning1982443}
under various circumstances where the mathematical expectation in
ARE (\ref{eq:riccati}) can be evaluated accurately.

A natural question is, how to solve LQ problem (\ref{eq:system})--(\ref{eq:cost})
when the statistical information of the parameters is inadequate and
ARE (\ref{eq:riccati}) fails to work. 

In this paper we propose a Q-learning algorithm to tackle this question.
Q-learning is a value-based reinforcement learning algorithm which
is used to find the optimal control policy using a state-control value
function, called the Q-function or Q-factor, instead of the usual
value function in dynamic programming; see~\cite{bertsekas2019reinforcement,sutton2018reinforcement}.
The original and most widely known Q-learning algorithm of Watkins
\cite{watkins1989learning} is a stochastic version of value iteration,
applying to Markov decision problems with unknown costs and transition
probabilities. The starting point is to reformulate Bellman's equation
into an equivalent form that is particularly convenient for deriving
learning algorithm. Let us briefly illustrate how this works in our
case. Defining
\[
Q^{*}=\E[\vN+\vA^{\tr}K\vA]\in\R^{d\times d},
\]
ARE (\ref{eq:riccati}) are equivalent to the following equation:
\begin{equation}
Q^{*}=\E[N+\vA^{\tr}\vpi(Q^{*})\vA].\label{eq:q-eq}
\end{equation}
The mathematical convenience of this reformulation derives from the
fact that the nonlinear operator $\vpi(\cdot)$ appears inside the
expectation in (\ref{eq:q-eq}), whereas it appears outside the expectation
in ARE (\ref{eq:riccati}). This fact plays an important role in the
feasibility and convergence of Q-learning algorithms.

\begin{algorithm}[!tb] 
\caption{Q-learning for LQ problem with random parameters} 
\label{alg:Q-learning} 
\begin{algorithmic}[1]
\STATE Set the initial matrix \(Q_0\). 
\WHILE{not converged}     
\STATE \(Q_{t+1} \gets Q_{t}+\alpha_{t}(\vN_{t+1}+\vA_{t+1}^{\tr}\vpi(Q_{t})\vA_{t+1}-Q_{t})\). 
\ENDWHILE 
\end{algorithmic} 
\end{algorithm}

Algorithm~\ref{alg:Q-learning} presented above is our Q-learning
algorithm for LQ problem (\ref{eq:system})--(\ref{eq:cost}), where
the learning rate sequence $(\alpha_{t}\in[0,1]:t=0,1,\dots)$ satisfies
\begin{equation}
\sum_{t=0}^\infty \alpha_{t}=\infty\quad\text{and}\quad
\sum_{t=0}^\infty \alpha_{t}^{2}<\infty,\label{eq:step-size-1}
\end{equation}
which can be random as long as $\alpha_{t}$ is measurable to $\sigma\{\vA_{s},\vN_{s}:s=1,\dots,t\}$.
The objective of this algorithm is to learn the matrix $Q^{*}$ (if
it exists) that solves (\ref{eq:q-eq}), based on the observed samples
of the parameters.

Algorithm~\ref{alg:Q-learning} is an online learning algorithm. As
pointed out by Tsitsiklis~\cite{tsitsiklis1994asynchronous}, the
Q-learning algorithm is recursive and each new piece of information
of the parameters is immediately used for computing an additive correction
term to the old estimates. The iteration in Algorithm~\ref{alg:Q-learning},
i.e.,
\begin{equation}
Q_{t+1}=Q_{t}+\alpha_{t}(\vN_{t+1}+\vA_{t+1}^{\tr}\vpi(Q_{t})\vA_{t+1}-Q_{t})\label{eq:q-learning}
\end{equation}
can be regarded as a stochastic version of the standard fixed-point
iteration based on the equation (\ref{eq:q-eq}).

The first theorem of this paper draws a full picture of the relationship
among the LQ problem, ARE, and the Q-learning algorithm. 
\begin{thm}
\label{thm:main}Let $(Q_{t})_{t\ge0}$ be the sequence constructed
in Algorithm~\ref{alg:Q-learning}. In addition to the above setting,
we assume that $\E[\|N\|_{2}^2+\|\vA^{\tr}\vA\|_{2}^2]$
is finite and $\E[\vN]$ is positive definite. 

Then, the following
statements are equivalent:
\begin{enumerate}
\item[\emph{a)}] LQ problem (\ref{eq:system})--(\ref{eq:cost}) is well-posed;
\item[\emph{b)}] ARE (\ref{eq:riccati}) admits a solution $K$;
\item[\emph{c)}] $(Q_{t})_{t\ge0}$ is bounded with a positive probability;
\item[\emph{d)}] $(Q_{t})_{t\ge0}$ converges almost surely (a.s.) to a deterministic
matrix $Q^{\star}\in\bS_{+}^{d}$.
\end{enumerate}
Moreover, if either statement is valid, one has the following properties:
\begin{enumerate}
\item[\emph{1)}] the value function $V(x)=x^{\tr}Kx$ for all $x\in\R^{n}$;
\item[\emph{2)}] the solution of ARE (\ref{eq:riccati}) is unique and given by $K=\vpi(Q^{\star})$;
\item[\emph{3)}] the optimal  control is given by  feedback form $u_{t}^{\star}=\Gamma(Q^{\star})x_{t}$;
\item[\emph{4)}] $Q^{\star}$ satisfies (\ref{eq:q-eq}), i.e., $Q^{\star}=\E[\vN+\vA^{\tr}\vpi(Q^{\star})\vA]$,
\end{enumerate}
where the mappings $\vpi(\cdot)$ and $\Gamma(\cdot)$ are defined
in (\ref{eq:pi}).
\end{thm}

In the theorem, we use $\|\cdot\|_2$ to denote the $2$-norm of Matrix. It is worth noting that statement (c) in this theorem looks relatively
weak, but it still implies the convergence of $Q_{t}$ and the well-posedness
of the LQ problem.  
Form above theorem, the sequence $Q_{t}$ is either convergent or divergent with probability 1. So we have
\begin{cor}
The probability that $(Q_{t})_{t\ge0}$ is bounded is either $0$
or $1$.
\end{cor}

This zero-one law endows the algorithm with great applicability. Indeed,
this proves, at least theoretically, that just running or stimulating
the system for one sample trajectory, we can almost certainly detect
whether the LQ problem is well-posed or not, and also obtain the desired
matrix $Q^{\star}$ if it exists. 

The next theorem concerns the stabilization of system (\ref{eq:system}).
We say system (\ref{eq:system}) is \emph{stabilizable} a.s. if under
some control the state vanishes a.s. as time tends to infinity. In
terms of the sequence $Q_{t}$ from Algorithm~\ref{alg:Q-learning},
it is natural to construct an adaptive feedback control
\begin{equation}
u_{t}^{\mathrm{a}}=\Gamma(Q_{t})x_{t}.\label{eq:fb-contr}
\end{equation}
It will be shown that this control stabilizes the
system a.s. as long as the LQ problem is well-posed.
\begin{thm}
\label{thm:stab}Under the same setting of Theorem~\ref{thm:main},
if LQ problem (\ref{eq:system})--(\ref{eq:cost}) is well-posed,
then the state $x_{t}^{{\rm a}}$ under the control $u_{t}^{\mathrm{a}}$
satisfies
\begin{equation}
J(x,u_{\cdot}^{\mathrm{a}})<\infty\quad\text{and}\quad\sum_{t=0}^{\infty}\big(|x_{t}^{\mathrm{a}}|^{2}+|u_{t}^{\mathrm{a}}|^{2}\big)<\infty\label{eq:2.7-2}
\end{equation}
a.s. for any initial state $x\in\R^{n}$; consequently, the adaptive
feedback control $u_{\cdot}^{\mathrm{a}}$ given by (\ref{eq:fb-contr})
stabilizes system (\ref{eq:system}) a.s.
\end{thm}

The stabilization in this result is in the path-wise sense, whereas
a relevant definition widely used in the literature is called the
mean-square stabilization, i.e., the second-order moment of the state,
$\E[\|x_{t}\|_{2}^{2}]$, tends to zero under certain control (see~\cite{DeKoning1982443,yaz1985stabilization}
from example). Either of these two definitions cannot cover each other,
while the path-wise one is relatively suitable in our setting because
our learning algorithm is carried out along one single sample path.
Nevertheless, certain modification of the adaptive feedback control
(\ref{eq:fb-contr}) may stabilize the system in these two sense simultaneously;
one possible way is to cut-off the feedback coefficient $\Gamma(Q_{t})$
when its norm is larger than some given bound. The details are left
to interested readers.

Let us give some remarks from the technical aspect. Like the original
Q-learning, our algorithm can also be embedded into a broad class
of stochastic approximation algorithms studied 
in~\cite{jaakkola1994convergence,tsitsiklis1994asynchronous},
two celebrated papers that give rigorous convergence proofs of the
original Q-learning. However, the results in those papers cannot apply
to Algorithm~\ref{alg:Q-learning} directly for two reasons: firstly,
they all require certain contraction conditions which are not satisfied
in our case, and secondly, the partial order of vectors used in~\cite{tsitsiklis1994asynchronous}
are substantially different from that of symmetric matrices, so the
monotonicity condition required there is not satisfied either. In
the proof of Theorem~\ref{thm:main} we adopt the comparison argument
from~\cite{tsitsiklis1994asynchronous}. The crucial fact we used
is the monotonicity property of $\vpi(\cdot)$, which helps us construct
upper and lower bounds for $Q_{t}$ from ARE on the one hand, and
an upper bound for the approximating sequence of ARE from $Q_{t}$
on the other hand. The equivalence of statements (a) and (b), i.e.,
the well-posdness of LQ problem and the solvability of ARE, is proved
by means of Bellman's principle of optimality; similar results can
be found in~\cite{DeKoning1982443,Morozan198389} where the control
is restricted in the feedback form.

The conditions of our results are quite general. The positive definiteness
of $\E[\vN]$ can be weakened to some extent (see Remark~\ref{rem:relax}),
which ensures the same property of the solution of ARE and the limit
$Q^{\star}$ (in this case the pseudoinverses in (\ref{eq:riccati})
and $\vpi(Q^{\star})$ are the standard matrix inverses). The finiteness
of $\E[\|N\|_{2}^2+\|\vA^{\tr}\vA\|_{2}^2]$ is a natural condition
to prove the convergence of Algorithm~\ref{alg:Q-learning} by use
of some classical results from stochastic approximation theory. In
applications, the verification of these conditions may depend on certain
qualitative properties of specific systems rather than the full statistical
information of parameters. For example, the conditions are automatically
satisfied if $\vA,\vN$ are bounded and $N$ is positive definite
a.s. Numerical examples are presented in Section 5. Nevertheless,
it would be interesting to extend the results to the indefinite case,
in which the matrix $N$ is not necessarily positive semidefinite.
The indefinite LQ problem has applications in many fields such as
robust control, mathematical finance, and so on; for more details,
we refer the reader to~\cite{ran1993linear,chen1998stochastic,rami2002indefinite,ni2015indefinite}. 

As far as LQ problems with unknown parameters are concerned, our approach
is different from the well-known adaptive control~\cite{aastrom2013adaptive}
in some respects. First, the types of randomness are different: the
noise in our model is multiplicative and has no specific structure,
whereas most of control systems studied in adaptive control are perturbed
by additive noise, for example, the linear--quadratic--Gaussian
control problem (see~\cite{chen1986optimal,duncan1999adaptive,faradonbeh2020adaptive,aastrom2013adaptive}
and references therein); as for multiplicative noise, the adaptive
control algorithms proposed in~\cite{alspach1973dual,tse1973actively},
of which the convergence are not proved, still specify the type of
noise and how it enters into the system. Second, a typical adaptive
control algorithm consists of two parts: parameter estimation (or
system identification) and control law, while in our approach we do
not pursuit to identify the system but directly learn the state-control
value function that yields the value function and control law. Nevertheless,
an adaptive control algorithm usually makes use of the inputs and
outputs of the system, but not the direct observation of the sample
of parameters as we do in this paper. Obviously, in a period $[t,t+1]$,
the information of inputs and outputs is often insufficient to determine
the exact value of the sample of parameters. Modification of our algorithm
based on inputs and outputs is expected in future work.

The rest of this paper is organized as follows. Section~2 presents some
auxiliary lemmas. The whole of Section~3 is devoted to
the proof of Theorem~\ref{thm:main}, split into five subsections.
Section~4 gives the proof of Theorem~\ref{thm:stab}. Numerical
examples and further discussion are presented in Section~5.

\section{Auxiliary lemmas}

Let us introduce some notations used in what follows. 
For a vector $x$ and a matrix $M$, $|x|$ and $\|M\|_{2}$ denote their $2$-norms, respectively. 
For two symmetric matrices $M_{1},M_{2}$
with the same size, we write $M_{1}\ge M_{2}$ (resp. $M_{1}>M_{2}$)
if $M_{1}-M_{2}$ is positive semidefinite (resp. definite); the notations
$M_{1}\le M_{2}$ and $M_{1}<M_{2}$ are similarly understood. $I_{n}$
denotes the $n\times n$ identity matrix, and $O$ the zero matrix
whose size is determined by the context.
\begin{lem}
\label{lem:Pi_monot}For any $Q_{1},Q_{2}\in\bS_{+}^{d}$, we have
\begin{gather}
\vpi(Q_{1}+Q_{2})\geq\vpi(Q_{1})+\vpi(Q_{2}).\label{eq:4.5-1}
\end{gather}
In particular, if $Q_{1}\le Q_{2}$ then $\vpi(Q_{1})\le\vpi(Q_{2})$.
\end{lem}

\begin{proof}
Recalling the definition of $\vpi(\cdot)$ in (\ref{eq:pi}), one
has that
\begin{equation}\label{eq:Pi-sense}
\min_{v\in\R^{m}}\begin{bmatrix}\begin{array}{l}
x\\
v
\end{array}\end{bmatrix}^{\tr}Q\begin{bmatrix}\begin{array}{l}
x\\
v
\end{array}\end{bmatrix}=x^{\tr}\vpi(Q)x,\quad\forall\,Q\in\bS_{+}^{d},\,x\in\R^{d},
\end{equation}
which implies 
\[
x^{\tr}\vpi(Q_{1}+Q_{2})x\geq x^{\tr}\vpi(Q_{1})x+x^{\tr}\vpi(Q_{2})x.
\]
So (\ref{eq:4.5-1}) is proved. 

Notice that $\vpi(Q)\ge O$ for any $Q\in\bS_{+}^{d}$. If $Q_{1}\le Q_{2}$,
then 
\[
\vpi(Q_{2}) \ge \vpi(Q_{1})+\vpi(Q_{2}-Q_{1})\ge\vpi(Q_{1}).
\]
This concludes the proof.
\end{proof}

The coming result about a deterministic recursion is elementary.
\begin{lem}
\label{lem:det-iter}Let $(f_{t}:t=0,1,2,\dots)$ be a bounded real valued sequence,
and $(\beta_{t})\subset[0,1]$ satisfy $\sum_{t}\beta_{t}=\infty$.
Suppose the sequence $y_{t}$ satisfies 
\[
y_{t+1}\le(1-\beta_{t})y_{t}+\beta_{t}f_{t}.
\]
Then
\[
\limsup_{t\to\infty}y_{t}\le\limsup_{t\to\infty}f_{t}.
\]
\end{lem}

\begin{proof}
Let $\tilde{f}_{t}=\sup_{s\ge t}f_{s}$. Then $\tilde{f}_{t}\ge f_{t}$
and $\tilde{f}_{t}\ge\tilde{f}_{t+1}$ for all $t=0,1,2,\dots$ Also,
define a sequence $(\tilde{y}_{t})$ with $\tilde{y}_{0}=|y_{0}|+|\tilde{f}_{0}|$
and
\begin{equation}
\tilde{y}_{t+1}=(1-\beta_{t})\tilde{y}_{t}+\beta_{t}\tilde{f}_{t}.\label{eq:tilde-y}
\end{equation}
It follows from induction that
\[
\tilde{y}_{t}\ge y_{t}\quad\text{and}\quad\tilde{f}_{t}\le\tilde{y}_{t}.
\]
Moreover, since $\tilde{y}_{t+1}$ is a convex combination of $\tilde{y}_{t}$
and $\tilde{f}_{t}$, one has $\tilde{y}_{t+1}\le\tilde{y}_{t}$.
Therefore, $\tilde{y}_{t}$ and $\tilde{f}_{t}$ are both decreasing
and convergent. Because the sequence
\[
z_{t}:=\tilde{y}_{0}-\tilde{y}_{t+1}=\sum_{s=0}^{t}\beta_{s+1}(\tilde{y}_{s}-\tilde{f}_{s})
\]
is uniformly bounded, plus the facts that $\tilde{y}_{s}-\tilde{f}_{s}$
is non-negative and $\sum_{s}\beta_{s}=\infty$, there is subsequence
$\{t'\}$ from $\{t\}$ such that
\[
\lim_{t'\to\infty}(\tilde{y}_{t'}-\tilde{f}_{t'})=0,
\]
which implies that $\tilde{y}_{t}$ and $\tilde{f}_{t}$ enjoy the same
limit. So
\[
\limsup_{t\to\infty}f_{t}=\lim_{t\to\infty}\tilde{f}_{t}=\lim_{t\to\infty}\tilde{y}_{t}\ge\limsup_{t\to\infty}y_{t}.
\]
The proof is complete.
\end{proof}

The following two results of stochastic approximation are important
in our argument.
\begin{lem}
\label{lem:4.2}Let $\{\mathcal{F}_{t}:t=0,1,2,\dots\}$ be a filtration
on a probability space $(\Omega,\mathcal{F},\Pb)$, and for each $t$,
let $r_{t}$ and $\alpha_{t}$ be $\mathcal{F}_{t}$-adapted scalar processes, where $\alpha_{t}\in[0,1]$ satisfies (\ref{eq:step-size-1}).
Suppose that there is an increasing sequence of stopping times $(\tau_{k}:k=1,2,\dots)$
such that
\begin{enumerate}
\item $\Pb(\sup_{k}\tau_{k}=\infty)>0$;
\item $\E[\mathbf{1}_{\{t<\tau_{k}\}}r_{t+1}|\mathcal{F}_{t\wedge\tau_{k}}]=0$
for all $t,k$;
\item $\E[\mathbf{1}_{\{t<\tau_{k}\}}r_{t+1}^{2}\big|\mathcal{F}_{t\wedge\tau_{k}}]\le\mu_{k}$
a.s. with a sequence of deterministic numbers $\mu_{k}$.
\end{enumerate}
Let $w_{t}$ satisfy the recursion: 
\[
w_{t+1}=(1-\alpha_{t})w_{t}+\alpha_{t}r_{t+1},\quad t=0,1,2\dots
\]
Then $\lim_{t\to\infty}w_{t}=0$ on $\{\sup_{k}\tau_{k}=\infty\}$
a.s.
\end{lem}

\begin{proof}
First of all, we assume that $\E[r_{t+1}|\mathcal{F}_{t}]=0$ and
$\E[r_{t+1}^{2}|\mathcal{F}_{t}]\leq\mu$ a.s. with some
number $\mu$. Then convergence of $w_{t}$ follows from some classical
results in stochastic approximation theory, e.g., Dvoretzky's extended
theorem~\cite{Dvoretzky1956}. 

Now for any $k\geq0$, define $r_{t}^{k}=\mathbf{1}_{\{t<\tau_{k}\}}r_{t}$
and $\mathcal{F}_{t}^{k}=\mathcal{F}_{t\wedge\tau_{k}}$. Then from
the assumptions and the above argument, the sequence $w_{t}^{k}$
constructed by
\[
w_{t+1}^{k}=(1-\alpha_{t})w_{t}^{k}+\alpha_{t}r_{t+1}^{k},\quad w_{0}^{k}=w_{0}
\]
converges to zero a.s. Notice that $w_{t}^{k}=w_{t}$ for all $t<\tau_{k}$,
so we have $\lim_{t\to\infty}w_{t}=0$ on $\{\sup_{k}\tau_{k}=\infty\}$
a.s.
\end{proof}

\begin{lem}
\label{lem:JJS}Let $\alpha_{t}$ be a $[0,1]$-valued process adapted
to a filtration $\mathcal{F}_{t}$ and satisfy (\ref{eq:step-size-1}),
and let $R_{t}$ be an $\mathcal{F}_{t}$-adapted process with
values in a vector space equipped with norm $\|\cdot\|$. Then the
iterative process
\[
X_{t+1}=(1-\alpha_{t})X_{t}+\alpha_{t}R_{t+1}
\]
converges to zero a.s. under the following assumptions:
\begin{enumerate}
\item $\|\E[R_{t+1}\vert\mathcal{F}_{t}]\|\le\lambda\|X_{t}\|$ a.s. with some
constant $\lambda<1$,
\item $\E[\|R_{t+1}\|^{2}\vert\mathcal{F}_{t}]\le\mu_{t}$ a.s., where $\mu_{t}\in\mathcal{F}_t$ and $\Pb\{\sup_t |\mu_t|<\infty\}=1$.
\end{enumerate}
\end{lem}

\begin{proof}
By means of stopping skill, it suffices to prove the lemma with $\mu_{k}$
dominated by a constant $\mu$. Let 
\[
S_{t+1}=R_{t+1}-\E[R_{t+1}\vert\mathcal{F}_{t}].
\]
Applying Lemma~\ref{lem:4.2}, one has that the vector sequence $W_{t}$
defined by
\[
W_{t+1}=(1-\alpha_{t})W_{t}+\alpha_{t}S_{t+1}
\]
converges to zero a.s. Setting $Y_{t}=X_{t}-W_{t}$, then
\begin{align*}
\|Y_{t+1}\| & =\|(1-\alpha_{t})Y_{t}+\alpha_{t}\E[R_{t+1}\vert\mathcal{F}_{t}]\|\\
 & \le(1-\alpha_{t})\|Y_{t}\|+\lambda\alpha_{t}\|Y_{t}+W_{t}\|\\
 & \le[1-(1-\lambda)\alpha_{t}]\|Y_{t}\|+(1-\lambda)\alpha_{t}\frac{\lambda}{1-\lambda}\|W_{t}\|.
\end{align*}
From Lemma~\ref{lem:det-iter} it follows that
\[
\limsup_{t\to\infty}\|Y_{t}\|\le\frac{\lambda}{1-\lambda}\lim_{t\to\infty}\|W_{t}\|=0\quad\text{a.s.}
\]
This concludes the lemma.
\end{proof}

\section{Proof of Theorem~\ref{thm:main}}

It is easily seen that (d) $\Rightarrow$ (c). In the coming five
subsections, we shall prove the relations (a) $\Rightarrow$ (b),
(b) $\Rightarrow$ (a), (b) $\Rightarrow$ (c), (c) $\Rightarrow$
(b), and (c) $\Rightarrow$ (d), respectively.

Let us do some preparations. Recall that $[\vA_{t}^{\tr},\vN_{t}]$
with $t=1,2,\dots$ are independent and identically distributed matrix-valued
random variables on a probability space $(\Omega,\mathcal{F},\Pb)$,
and $[\vA^{\tr},\vN]$ denotes an independent copy of $[\vA_{1}^{\tr},\vN_{1}]$.
Introduce the filtration $(\mathcal{F}_{t})$ with $\mathcal{F}_{0}=\{\varnothing,\Omega\}$
and 
\[
\mathcal{F}_{t}=\sigma\{\vA_{s},\vN_{s}:s=1,\dots,t\},\quad t=1,2,\dots.
\]
Also recall that the control sequence is allowed to be any $\mathcal{F}_{t}$-adapted
process in $\R^{m}$ (not necessarily of feedback form). 

It is convenient to define two mappings 
\begin{equation}
\begin{aligned}\F_{t}(Q) & :=\vN_{t}+\vA_{t}^{\tr}\vpi(Q)\vA_{t},\\
\F(Q) & :=\E[\vN+\vA^{\tr}\vpi(Q)\vA]
\end{aligned}
\label{eq:phi}
\end{equation}
for $Q\in\bS_{+}^{d}$; apparently, $\F(Q)=\E[\F_{t}(Q)]$. 

An important fact is that $\F_{t}(\cdot)$ and $\F(\cdot)$ are both
increasing, i.e., if $Q_{1}\le Q_{2}$ then $\F_{t}(Q_{1})\le\F_{t}(Q_{2})$
and $\F(Q_{1})\le\F(Q_{2})$. This follows from the monotonicity of
$\vpi(\cdot)$, see Lemma~\ref{lem:Pi_monot}.

Occasionally, $\vA_{t}\in\R^{n\times d}$ is written into a block
matrix $\vA_{t}=[A_{t},B_{t}]$ with $A_{t}\in\R^{n\times n}$, then
the system reads
\begin{equation}
x_{t+1}=A_{t+1}x_{t}+B_{t+1}u_{t}.\label{eq:systems}
\end{equation}
From the condition of Theorem~\ref{thm:main}, there are two numbers
$\eps_{0}\in(0,1]$ and $\mu_{0}>0$ such that
\begin{equation}
\E[N+\vA^{\tr}\vA]\ge\eps_{0}I_{d},\quad\E[\|N\|_{2}^{2}+\|\vA^{\tr}\vA\|_{2}^{2}]\le\mu_{0}.\label{eq:eps}
\end{equation}

Finally, we claim that, if $u_{t}\in L^{2}(\Omega)$ for all $t$,
then $x_{t}\in L^{2}(\Omega)$ for all $t$. Indeed, assuming $x_{t}\in L^{2}(\Omega)$
we compute
\begin{align*}
\E[|x_{t+1}|^{2}\vert\mathcal{F}_{t}] & =\E\bigg\{\big[x_{t}^{\tr},u_{t}^{\tr}\big]\vA_{t+1}^{\tr}\vA_{t+1}\begin{bmatrix}x_{t}\\
u_{t}
\end{bmatrix}\bigg|\mathcal{F}_{t}\bigg\}\\
 & =\big[x_{t}^{\tr},u_{t}^{\tr}\big](\E[\vA_{t+1}^{\tr}\vA_{t+1}])\begin{bmatrix}x_{t}\\
u_{t}
\end{bmatrix}\\
 & \le c(\mu_{0})(|x_{t}|^{2}+|u_{t}|^{2}),
\end{align*}
thus, $\E[\|x_{t+1}\|^{2}]<\infty$. The claim is so verified by induction.
This ensures the well-posedness of the LQ problem in finite horizon.

\subsection{\label{sec:SLQR-ARE}From LQ problem to ARE}

Consider the optimal control in finite horizon: let $T$ be a large
natural number and define
\[
V_{T}(x,t)=\inf_{u_{t},\dots,u_{T-1}}\E\bigg\{\sum_{s=t}^{T-1}\big[x_{s}^{\tr},u_{s}^{\tr}\big]\vN_{s+1}\begin{bmatrix}x_{s}\\
u_{s}
\end{bmatrix}\,\bigg|\,x_{t}=x\bigg\}.
\]
Recalling (\ref{eq:value}) the value function $V(\cdot)$ of the
original problem, it is easily seen that 
\[
V_{T}(x,t)\le V(x)\quad\forall\,x\in\R^{n},\,t=0,1,\dots,T-1.
\]
For this finite horizon problem, it follows from Bellman's principle
of optimality (cf. {[}Aoki, p. 32{]}) that
\begin{align*}
V_{T}(x,t) & =\inf_{u_{t}}\E\bigg\{\big[x_{t}^{\tr},u_{t}^{\tr}\big]\vN_{t+1}\begin{bmatrix}x_{t}\\
u_{t}
\end{bmatrix}+V_{T}(x_{t+1},t+1)\,\bigg|\,x_{t}=x\bigg\}
\end{align*}
for $t=0,1,\dots,T-1$; in particular, we set $V_{T}(x,T)=0$. Assume
that $V_{T}(x,t+1)$ is a quadratic form for some $t\le T-1$, namely,
\[
V_{T}(x,t+1)=x^{\tr}\mathcal{K}_{t+1,T}x,\quad\text{where }\mathcal{K}_{t+1,T}\text{ is a symmetric matrix}.
\]
Then $V_{T}(x,t)$ is also a quadratic form:
\begin{align*}
V_{T}(x,t) & =\inf_{u_{t}}\E\bigg\{\big[x_{t}^{\tr},u_{t}^{\tr}\big]\vN_{t+1}\begin{bmatrix}x_{t}\\
u_{t}
\end{bmatrix}+\big[x_{t}^{\tr},u_{t}^{\tr}\big]\vA_{t+1}^{\tr}\mathcal{K}_{t+1,T}\vA_{t+1}\begin{bmatrix}x_{t}\\
u_{t}
\end{bmatrix}\,\bigg|\,x_{t}=x\bigg\}\\
 & =\inf_{u_{t}}\E\bigg\{\big[x^{\tr},u_{t}^{\tr}\big]\big(\vN_{t+1}+\vA_{t+1}^{\tr}\mathcal{K}_{t+1,T}\vA_{t+1}\big)\begin{bmatrix}x\\
u_{t}
\end{bmatrix}\bigg\}\\
 & =\inf_{u\in\R^{m}}\big[x^{\tr},u^{\tr}\big]\E[\vN_{t+1}+\vA_{t+1}^{\tr}\mathcal{K}_{t+1,T}\vA_{t+1}]\begin{bmatrix}x\\
u
\end{bmatrix}\\
 & =x^{\tr}\Pi(\E[\vN_{t+1}+\vA_{t+1}^{\tr}\mathcal{K}_{t+1,T}\vA_{t+1}])x\\
 & =:x^{\tr}\mathcal{K}_{t,T}x.
\end{align*}
By induction, one has
\begin{equation}
V_{T}(x,t)=x^{\tr}\mathcal{K}_{t,T}x\quad\forall\,t=0,1,\dots,T,\label{eq:V-K}
\end{equation}
where the matrices $\mathcal{K}_{t,T}$ satisfy the following algebraic
Riccati equations:
\begin{equation}
\mathcal{K}_{t,T}=\Pi(\E[\vN+\vA^{\tr}\mathcal{K}_{t+1,T}\vA]),\quad\mathcal{K}_{T,T}=O.\label{eq:K-finite}
\end{equation}
The following result shows that ARE (\ref{eq:riccati}) has a solution
as long as the LQ problem is well-posed.
\begin{prop}
\label{prop:LQ-ARE}Define a sequence $\{K_{t}:t=0,1,2,\dots\}$ recursively
as follows:
\begin{equation}
\begin{aligned}K_{0} & =O,\\
K_{t+1} & =\vpi(\E[\vN+\vA^{\tr}K_{t}\vA]).
\end{aligned}
\label{eq:K}
\end{equation}
If LQ problem (\ref{eq:system})--(\ref{eq:cost}) is well-posed,
then $K_{t}$ converges to a matrix $K$ that solves ARE (\ref{eq:riccati}).
Moreover, the solution $K$ obtained here is the minimum solution
of ARE (\ref{eq:riccati}), i.e., $K\le\tilde{K}$ if $\tilde{K}$
also satisfies ARE (\ref{eq:riccati}).
\end{prop}

\begin{proof}
Comparing (\ref{eq:K-finite}) and (\ref{eq:K}), it is easily seen
that $K_{t}=\mathcal{K}_{T-t,T}$ for any $T>t$, which along with (\ref{eq:V-K})
yields $x^{\tr}K_{T-t}x=x^{{\rm T}}\mathcal{K}_{t,T}x=V_{T}(x,t)$.
According to its definition, $V_{T}(x,0)$ is increasing in $T$,
so $K_{T}$ is also increasing. Therefore, if LQ problem (\ref{eq:system})--(\ref{eq:cost})
is well-posed, then for any unit $x\in\R^{n}$,
\begin{equation}
x^{\tr}K_{T}x=V_{T}(x,0)\le V(x)<\infty,\label{eq:K<V}
\end{equation}
which means that $\{x^{\tr}K_{t}x:t\ge0\}$ is uniformly bounded,
as $K_{t}$ is increasing, this fact implies that $K_{t}$ converges
to a matrix, denoted by $K$. From the recursive relation of $K_{t}$,
one has that $K$ obtained satisfies ARE. Moreover, (\ref{eq:K<V})
implies that
\begin{equation}
x^{\tr}Kx\le V(x)\quad\forall\,x\in\R^{n}.\label{eq:K<V-2}
\end{equation}
Let $\tilde{K}$ be any solution of ARE (\ref{eq:riccati}). Since
$K_{0}=O\le\tilde{K}$, it follows from induction and the monotonicity
of $\Pi(\cdot)$ that $K_{t}\le\tilde{K}$, which implies the limit
$K$ is also dominated by $\tilde{K}$. The proof is complete.
\end{proof}

\subsection{From ARE to LQ problem}

Assume that ARE (\ref{eq:riccati}) has a solution $K$. Select the
control
\begin{equation}
u_{t}=\varGamma x_{t}\ \ \text{ with }\ \varGamma=\Gamma(\E[\vN+\vA^{\tr}K\vA]),\label{eq:gamma}
\end{equation}
where $\Gamma(\cdot)$ is defined in (\ref{eq:pi}). Recalling (\ref{eq:systems}),
the system becomes
\[
x_{t+1}=(A_{t+1}+B_{t+1}\varGamma)x_{t}=:\bar{A}_{t+1}x_{t},
\]
and the cost function reads
\[
J(x,\varGamma x_{\cdot})=\sum_{t=0}^{\infty}x_{t}^{\tr}\big[I_{n},\varGamma^{\tr}\big]\vN_{t+1}\begin{bmatrix}I_{n}\\
\varGamma
\end{bmatrix}x_{t}=:\sum_{t=0}^{\infty}x_{t}^{\tr}\bar{M}_{t+1}x_{t}.
\]
To show that $\E[J(x,\varGamma x_{\cdot})]<\infty$, we define
\[
\bar{V}_{T}(x,t)=\E\bigg\{\sum_{s=t}^{T-1}x_{s}^{\tr}\bar{M}_{s+1}x_{s}\,\bigg|\,x_{t}=x\bigg\}.
\]
With a similar argument as in the last subsection, one can show that
there is a symmetric matrix $\bar{\mathcal{K}}_{t,T}$ such that $x^{\tr}\bar{\mathcal{K}}_{t,T}x=\bar{V}_{T}(x,t)$,
and
\[
\bar{\mathcal{K}}_{t,T}=\E[\bar{M}_{t+1}+\bar{A}_{t+1}^{\tr}\bar{\mathcal{K}}_{t+1,T}\bar{A}_{t+1}],\quad\bar{\mathcal{K}}_{T,T}=O.
\]
Direct computation gives that
\[
\E[\bar{M}_{t+1}+\bar{A}_{t+1}^{\tr}\bar{\mathcal{K}}_{t+1,T}\bar{A}_{t+1}]=\Pi(\E[\vN_{t+1}+\vA_{t+1}^{\tr}\bar{\mathcal{K}}_{t+1,T}\vA_{t+1}]),
\]
which implies $\bar{\mathcal{K}}_{t,T}=\mathcal{K}_{t,T}$, where
$\mathcal{K}_{t,T}$ is defined in (\ref{eq:V-K}). Thus, one has
\[
\bar{V}_{T}(x,0)=x^{\tr}\bar{\mathcal{K}}_{0,T}x=x^{\tr}\mathcal{K}_{0,T}x=x^{\tr}K_{T}x,
\]
where $K_{T}$ is defined in (\ref{eq:K}). It follows from induction
that $K_{T}\le K$, so
\begin{equation}
V(x)\le\E[J(x,\varGamma x_{\cdot})]\le\limsup_{T\to\infty}\bar{V}_{T}(x,0)\le x^{\tr}Kx<\infty\quad\forall\,x\in\R^{n}.\label{eq:V<K}
\end{equation}
Therefore, the LQ problem is well-posed.
\begin{rem}
\label{rem:optimal}If ARE has a unique solution $K$ (which
is proved in Lemma~\ref{lem:ARE-unique}), then it follows from
(\ref{eq:K<V-2}) and (\ref{eq:V<K}) that $V(x)=x^{\tr}Kx$, and
consequently, $V(x)=\E[J(x,\varGamma x_{\cdot})]$ for all $x\in\R^{n}$.
This means that the feedback control $u_{t}=\Gamma(\E[\vN+\vA^{\tr}K\vA])x_{t}$
is optimal.
\end{rem}

\begin{rem}
The assumption that $\E[\vN]$ is positive definite is not used above
to prove the equivalence of statements (a) and (b), i.e., the well-posedness
of LQ problem and the solvability of ARE. Off course, if stament (a) or (b) holds, then it must have that $\E[\vN]$ is non-negative definite. But it allows $\E[\vN]$ to be degenerate. 
\end{rem}

\subsection{From ARE to boundedness of $Q_{t}$}

Assume that ARE (\ref{eq:riccati}) has a solution $K$. 
The basic idea of this part is to transform the problem into an equivalent form
in which the solution of ARE becomes the identity matrix. Denote 
\begin{align}
Q^{*} & :=\E[\vN+\vA^{\tr}K\vA],\label{eq:Q*}\\
\varGamma & :=\Gamma(Q^{*})=-(Q_{uu}^{*})^{-1}Q_{ux}^{*}.\nonumber 
\end{align}
Since $\E[N]\ge\eps_{0}I_{d}$, one has 
\begin{equation}
K=\vpi(\E[\vN+\vA^{\tr}K\vA])\ge\vpi(\E[\vN])\ge\eps_{0}I_{n}>O.\label{eq:K-posit}
\end{equation}
Similarly, $\E[Q^{*}]\ge\eps_{0}I_{d}$. Let $L$ and $M$ be invertible
matrices such that
\[
L^{\tr}KL=I_{n},\quad M^{\tr}Q_{uu}^{*}M=I_{m}.
\]
Introducing
\[
C:=\begin{bmatrix}I_{n} & O\\
\varGamma & I_{m}
\end{bmatrix}\left[\begin{array}{cc}
L & O\\
O & M
\end{array}\right]=\begin{bmatrix}L & O\\
\varGamma L & M
\end{bmatrix},
\]
we define
\begin{align*}
\tilde{\vA}_{t} & :=L^{-1}\vA_{t}C,\\
\tilde{\vN}_{t} & :=C^{\tr}\vN_{t}C.
\end{align*}
It is easily verified that
\[
I_{n}=\Pi(\E[\tilde{\vN}+\tilde{\vA}^{\tr}\tilde{\vA}]).
\]
This means, after transformation, the solution of (new) ARE is the
identity matrix $I_{n}$.

Now we reformulate the Q-learning process. Let
\[
\tilde{Q}_{t}:=C^{\tr}Q_{t}C.
\]
We define
\begin{align*}
\tilde{\F}_{t}(Q) & :=\tilde{\vN}_{t}+\tilde{\vA}_{t}^{\tr}\vpi(Q)\tilde{\vA}_{t},\\
\tilde{\F}(Q) & :=\E[\tilde{\F}_{t}(Q)]=\E[\tilde{\vN}+\tilde{\vA}^{\tr}\vpi(Q)\tilde{\vA}].
\end{align*}
Observing 
\begin{equation}
\begin{aligned}C^{\tr}QC & =\begin{bmatrix}\begin{gathered}L^{\tr}(Q_{xu}\varGamma+\varGamma^{\tr}Q_{ux}\\
+Q_{xx}+\varGamma^{\tr}Q_{uu}\varGamma)L
\end{gathered}
 & L^{\tr}(Q_{xu}+\varGamma^{\tr}Q_{uu})M\\
\\
M^{\tr}(Q_{ux}+Q_{uu}\varGamma)L & M^{\tr}Q_{uu}M
\end{bmatrix},\end{aligned}
\label{eq:4.8}
\end{equation}
 one can check that 
\begin{equation}
\vpi(C^{\tr}QC)=L^{\tr}\vpi(Q)L.\label{eq:4.3-6}
\end{equation}
Under the above transformation, the iteration in Algorithm~\ref{alg:Q-learning}
is equivalently written into
\[
\begin{aligned}\tilde{Q}_{0} & =C^{\tr}Q_{0}C,\\
\tilde{Q}_{t+1} & =\tilde{Q}_{t}+\alpha_{t}(\tilde{\F}_{t+1}(\tilde{Q}_{t})-\tilde{Q}_{t}).
\end{aligned}
\]
According to $Q^{*}=\F(Q^{*})$, one can see that 
\begin{equation}
I_{d}=\tilde{\F}(I_{d}).\label{eq:Id}
\end{equation}
In other words, $I_{d}$ is a fixed point of $\tilde{\F}(\cdot)$. 

Define \emph{affine} mappings
\begin{equation}
\Psi_{t}(Q):=\tilde{\vN}_{t}+\tilde{\vA}_{t}^{\tr}Q_{xx}\tilde{\vA}_{t}\ge\tilde{\F}_{t}(Q)\label{eq:Psi_t}
\end{equation}
and
\begin{equation}
\Psi(Q):=\E[\Psi_{t}(Q)]\ge\tilde{\F}(Q).\label{eq:4.2-2}
\end{equation}
Comparing to $\tilde{\F}(\cdot)$, the most important property of
$\Psi(\cdot)$ is that $\Psi(\cdot)$ is a contraction mapping under
matrix $2$-norm. To see this, one first obtains from (\ref{eq:Id}) and
(\ref{eq:4.2-2}) that
\begin{equation}
I_{d}=\Psi(I_{d})=\E[\tilde{\vN}+\tilde{\vA}^{\tr}\tilde{\vA}].\label{eq:psi_eq}
\end{equation}
$\E[\vN]>O$ implies $\E[\tilde{\vN}]>O$, so there is a positive
number $\lambda<1$ such that
\begin{equation}
\E[\tilde{\vA}^{\tr}\tilde{\vA}]=I_{d}-\E[\tilde{N}]\le\lambda I_{d}.\label{eq:lambda}
\end{equation}
Using an elementary fact from linear algebra: $S^{\tr}TS\le\|T\|_{2}\,S^{\tr}S$,
where $T$ is a symmetric matrix, one has that for any $Q_{1},Q_{2}\in\bS_{+}^{d}$,
\begin{align*}
\|\Psi(Q_{1})-\Psi(Q_{2})\|_{2} & =\|\E[\tilde{\vA}^{\tr}(Q_{1,xx}-Q_{2,xx})\tilde{\vA}]\|_{2}\\
 & \le\|Q_{1,xx}-Q_{2,xx}\|_{2}\|\E[\tilde{\vA}^{\tr}\tilde{\vA}]\|_{2}\\
 & \le\lambda\|Q_{1}-Q_{2}\|_{2}.
\end{align*}
Since $\lambda<1$, $\Psi(\cdot)$ is a contraction mapping.

Introduce the iteration:
\begin{equation}
\begin{aligned}P_{0} & =C^{\tr}Q_{0}C,\\
P_{t+1} & =P_{t}+\alpha_{t}(\Psi_{t+1}(P_{t})-P_{t}).
\end{aligned}
\label{eq:P-process}
\end{equation}
From (\ref{eq:Psi_t}) one can see that
\begin{equation}
\tilde{Q}_{t}\le P_{t}\quad\forall\,t=0,1,2,\dots\label{eq:compar}
\end{equation}
So $P_{t}$ is an upper bound process of $\tilde{Q}_{t}$. 
\begin{lem}
\label{lem:P-bdd}Under the above setting, $P_{t}$ converges to $I_{d}$
a.s. Consequently, the sequence $Q_{t}$ is bounded a.s.
\end{lem}

\begin{proof}
Define $X_{t}=P_{t}-I_{d}$. Using the relation $I_{d}=\Psi(I_{d})$,
one has
\begin{align*}
X_{t+1}= & (1-\alpha_{t})X_{t}+\alpha_{t}R_{t+1}
\end{align*}
with
\[
R_{t+1}:=\Psi_{t+1}(P_{t})-\Psi(P_{t})+\E[\vA^{\tr}(P_{t,xx}-I_{n})\vA].
\]

Now we check that
\begin{align*}
\E[R_{t+1}\vert\mathcal{F}_{t}] & =\E[\Psi_{t+1}(P_{t})-\Psi(P_{t})\vert\mathcal{F}_{t}]+\E[\vA^{\tr}(P_{t,xx}-I_{n})\vA]\\
 & =\E[\vA^{\tr}(P_{t,xx}-I_{n})\vA]\leq \|P_{t,xx}-I_{n}\|_{2} \E[\vA^{\tr}\vA]\\
 & \le\lambda\|P_{t,xx}-I_{n}\|_{2}I_{d}\le\lambda\|X_{t}\|_{2}I_{d},
\end{align*}
which implies $\|\E[R_{t+1}\vert\mathcal{F}_{t}]\|_{2}\le\lambda\|X_{t}\|_{2}$.
Moreover, we have
\begin{align*}
\|R_{t+1}\|_{2}^{2} & \le3(\|\Psi_{t+1}(P_{t})\|_{2}^{2}+\|\Psi(P_{t})\|_{2}^{2}+\|\E[\vA^{\tr}(P_{t,xx}-I_{n})\vA]\|_{2}^{2})\\
 & \le6(\|N_{t+1}\|_{2}^{2}+\|\vA_{t+1}^{\tr}\vA_{t+1}\|_{2}^{2}\|P_{t}\|_{2}^{2}+\|\E[N]\|_{2}^{2}+\|\E[\vA^{\tr}\vA]\|_{2}^{2}\|P_{t}\|_{2}^{2}\\
 & \quad\quad+\|\E[\vA^{\tr}\vA]\|_{2}^{2}\|X_{t}\|_{2}^{2}).
\end{align*}
Since $\vA_{t+1}$ is independent of $\mathcal{F}_{t}$, 
\begin{align*}
\E[\|\vA_{t+1}^{\tr}\vA_{t+1}\|_{2}^{2}\|P_{t}\|_{2}^{2}\vert\mathcal{F}_{t}] & =\|P_{t}\|_{2}^{2}\,\E[\|\vA_{t+1}^{\tr}\vA_{t+1}\|_{2}^{2}\vert\mathcal{F}_{t}]\\
 & =\|P_{t}\|_{2}^{2}\,\E[\|\vA^{\tr}\vA\|_{2}^{2}].
\end{align*}
Using the relation $\|P_{t}\|_{2}\le1+\|X_{t}\|_{2}$ and recalling
the constant $\mu_{0}$ from (\ref{eq:eps}), one obtains
\[
\E[\|R_{t+1}\|_{2}^{2}\vert\mathcal{F}_{t}]\le12\mu_{0}+12\mu_{0}\|P_{t}\|_{2}^{2}+6\mu_{0}\|X_{t}\|_{2}^{2}\le36\mu_{0}+30\mu_{0}\|X_{t}\|_{2}^{2}.
\]
Therefore, to apply Lemma~\ref{lem:JJS} it suffices to prove that
the sequence $X_{t}$ (or equivalently, $P_{t}$) is bounded a.s.
The proof is quite similar to that of~\cite[Theorem~1]{tsitsiklis1994asynchronous}.
For completeness, we give the details as follows.

Let $\eps>0$ satisfy $\lambda(1+\eps)=1$, and
$\tilde{m}_t := 1 + \max_{s \le t}\|X_{s}\|_{2}$.
Define a sequence $m_t \in \mathcal{F}_{t}$ recursively:
$m_0 = 1$ and
\[
m_{t+1}=\bigg{\{}
\begin{aligned}
 & m_t, & & \text{if }\  \tilde{m}_{t+1} \le m_t (1+\eps),\\
 & \min\{ (1+\eps)^k : \tilde{m}_{t+1} \le (1+\eps)^k \}, & &  \text{if }\  \tilde{m}_{t+1} > m_t (1+\eps),
\end{aligned}
\]
where $k$ is an integer. 
Notice that $\|X_t\|_2\le m_t(1+\eps)$, and $\|X_t\|_2\le m_t$ if $m_t > m_{t-1}$.

Define $S_{t+1}=m_{t}^{-1}(\Psi_{t+1}(P_{t})-\Psi(P_{t})).$  Obviously, $\E[S_{t+1}|\mathcal{F}_t]=0$; and by above arguments, $\E[\|S_{t+1}\|_2^2|\mathcal{F}_t]<C$ a.s. for all $t$ and some constant $C>0$. Then it follows
from Lemma~\ref{lem:JJS} that the sequence $Z_{t}$ defined by
\[
Z_{t+1}=(1-\alpha_{t})Z_{t}+\alpha_{t}S_{t+1}
\]
converges to zero matrix, so there is a full probability set $\Omega'\subset\Omega$
and a random time $T(\omega)$ for each $\omega\in\Omega'$ such that
$\|Z_{t}(\omega)\|_{2}<\eps/2$ for all $t\ge T(\omega)$. 

Fix an $\omega\in\Omega'$. 
If $m_t(\omega) = m_{T}(\omega)$ for all $t>T(\omega)$, then $\|X_t(\omega)\|$ is bounded 
by $m_T(\omega) (1+\eps)$, the proof is so concluded.
Otherwise, there is a $\tau > T(\omega)$ such that $m_{{\tau}} > m_{{\tau}-1}$, 
thus $\|X_{{\tau}}\|_2 \le m_{{\tau}}$.
Define
$Z_{t}^{\tau}$ with $Z_{\tau}^{\tau}=O$ and
\[
Z_{t+1}^{\tau}=(1-\alpha_{t})Z_{t}^{\tau}+\alpha_{t}S_{t+1},\quad t=\tau,\tau+1,\dots
\]
Notice that for $t\ge\tau+1$,
\[
Z_{t}=\Big[\prod_{s=\tau}^{t-1}(1-\alpha_{s})\Big]Z_{\tau}+Z_{t}^{\tau},
\]
so $\|Z_{t}^{\tau}(\omega)\|_{2}\le\|Z_{\tau}(\omega)\|_{2}+\|Z_{t}(\omega)\|_{2}<\eps$.
Next, we prove by induction that 
\begin{equation}
X_{t}\le m_{\tau}(I_{d}+Z_{t}^{\tau})<m_{\tau}(1+\eps)I_{d},\quad t=\tau,\tau+1,\dots\label{eq:P-bound}
\end{equation}
This holds true when $t=\tau$. Assume that it holds for $\tau,\tau+1,\dots,t$.
Under this assumption, one knows that
$m_{\tau}=m_{\tau+1}=\cdots=m_{t}$.
Defining
\[
\hat{\Psi}(Q):=\E[\vA^{\tr}Q_{xx}\vA]\le\lambda\|Q\|_{2}I_{d},
\]
we compute (recalling that $\lambda(1+\eps)=1$)
\begin{align*}
X_{t+1} & =(1-\alpha_{t})X_{t}+\alpha_{t}\hat{\Psi}(X_{t})+\alpha_{t}m_{t}S_{t+1}\\
 & \le(1-\alpha_{t})m_{\tau}(I_{d}+Z_{t}^{\tau})+\alpha_{t}\lambda\|X_{t}\|_{2}I_{d}+\alpha_{t}m_{\tau}S_{t+1}\\
 & <(1-\alpha_{t})m_{\tau}(I_{d}+Z_{t}^{\tau})+\alpha_{t}\lambda m_{\tau}(1+\eps)I_{d}+\alpha_{t}m_{\tau}S_{t+1}\\
 & =(1-\alpha_{t})m_{\tau}(I_{d}+Z_{t}^{\tau})+\alpha_{t}m_{\tau}I_{d}+\alpha_{t}m_{\tau}S_{t+1}\\
 & =m_{\tau}[I_{d}+(1-\alpha_{t})Z_{t}^{\tau}+\alpha_{t}S_{t+1}]\\
 & =m_{\tau}(I_{d}+Z_{t+1}^{\tau}).
\end{align*}
Hence, (\ref{eq:P-bound}) holds true. This implies that the sequence
$X_{t}(\omega)$ is bounded. The proof is complete.
\end{proof}

The above transformation can also help us prove uniqueness of the
solution of ARE (\ref{eq:riccati}).
\begin{lem}
\label{lem:ARE-unique}ARE (\ref{eq:riccati}) has at most one solution.
\end{lem}

\begin{proof}
If ARE (\ref{eq:riccati}) is solvable, then it has a minimum solution
$K$ in view of Proposition~\ref{prop:LQ-ARE}. Using this solution
to do the transformation above, then $I_{n}$ is the minimum solution
to the following equation for $\tilde{K}$:
\[
\tilde{K}=\Pi(\E[\tilde{\vN}+\tilde{\vA}^{\tr}\tilde{K}\tilde{\vA}]).
\]
Equivalently, $I_{d}$ is the minimum fixed point of $\tilde{\F}(\cdot)$.
Let $\tilde{Q}$ be any fixed point of $\tilde{\F}(\cdot)$, then
\begin{align*}
\tilde{Q}-I_{d} & =\tilde{\Phi}(\tilde{Q})-I_{d}\le\Psi(\tilde{Q})-I_{d}=\Psi(\tilde{Q})-\Psi(I_{d})\\
 & =\E[\vA^{\tr}(\tilde{Q}-I_{d})\vA]\le\|\tilde{Q}-I_{d}\|_{2}\,\E[\vA^{\tr}\vA].
\end{align*}
From (\ref{eq:lambda}) one has that
\[
\|\tilde{Q}-I_{d}\|_{2}\le\|\E[\vA^{\tr}\vA]\|_{2}\|\tilde{Q}-I_{d}\|_{2}\le\lambda\|\tilde{Q}-I_{d}\|_{2}.
\]
As $\lambda<1$, this means $\|\tilde{Q}-I_{d}\|_{2}=0$, so $\tilde{Q}=I_{d}$,
implying uniqueness of the fixed point of $\tilde{\F}(\cdot)$, and
furthermore, uniqueness of the solution of ARE (\ref{eq:riccati}).
The proof is complete.
\end{proof}

Let $K$ be the unique solution of ARE (\ref{eq:riccati}). Then it
follows from (\ref{eq:K<V-2}) and (\ref{eq:V<K}) that $V(x)=x^{\tr}Kx$
for all $x\in\R^{n}$, which proves property (1) in Theorem~\ref{thm:main}. 

\subsection{From boundedness of $Q_{t}$ to ARE}

Recall that $\E[N]\ge\eps_{0}I_{d}>O$. For each $\varepsilon\in[0,\varepsilon_{0}/2)$
, we define recursively a sequence of deterministic matrices:
\begin{equation}
\begin{aligned}L_{0}^{\varepsilon} & =O,\\
L_{k+1}^{\varepsilon} & =\F(L_{k}^{\varepsilon})-\eps I_{d},\quad k=0,1,2\dots
\end{aligned}
\label{eq:Leps-iter}
\end{equation}

\begin{lem}
\label{lem:Leps}$L_{k}^{\varepsilon}\le L_{k+1}^{\varepsilon}$ for
all $k=0,1,2,\dots$
\end{lem}

\begin{proof}
It can be proved easily by induction. First, we have that $L_{1}^{\eps}=\F(O)-\eps I_{d}=\E[\vN]-\eps I_{d}>O=L_{0}^{\eps}$.
Assume that $L_{k-1}^{\eps}\le L_{k}^{\eps}$. Then from the monotonicity
of $\F(\cdot)$, one has
\[
L_{k+1}^{\varepsilon}=\F(L_{k}^{\varepsilon})-\eps I_{d}\ge\F(L_{k-1}^{\varepsilon})-\eps I_{d}=L_{k}^{\varepsilon}.
\]
The lemma is proved.
\end{proof}

\begin{lem}
\label{prop:comp}Define $\Theta:=\{\omega:\sup_{t}\|Q_{t}\|_{2}<\infty\}$.
If $\Pb(\Theta)>0$, then for each $\varepsilon\in(0,\varepsilon_{0}/2)$,
there is a sequence of increasing random times $t_{k}$ such that
for almost all $\omega\in\Theta$,
\begin{gather}
L_{k}^{\varepsilon}\le Q_{t}(\omega)\quad\forall t\geq t_{k}.\label{eq:4.9-2}
\end{gather}
\end{lem}

\begin{proof}
We rewrite iteration (\ref{eq:q-learning}) as
\begin{align*}
Q_{t+1} & =(1-\alpha_{t})Q_{t}+\alpha_{t}\F_{t+1}(Q_{t})\\
 & =(1-\alpha_{t})Q_{t}+\alpha_{t}(\F(Q_{t})+R_{t+1}(Q_{t})),
\end{align*}
where
\[
R_{t+1}(Q):=\F_{t+1}(Q)-\F(Q).
\]
Now we introduce a decomposition of above iteration. For $t=0,1,2,\dots$,
define 
\[
\begin{aligned}Y_{0} & =O,\\
Y_{t+1} & =(1-\alpha_{t})Y_{t}+\alpha_{t}\F(Q_{t});\\
W_{0} & =O,\\
W_{t+1} & =(1-\alpha_{t})W_{t}+\alpha_{t}R_{t+1}(Q_{t}).
\end{aligned}
\]
It is easy to know that $Q_{t}=Y_{t}+W_{t}$. As $N_{t+1}$ and $\vA_{t+1}$
are independent of $\mathcal{F}_{t}$, one can check that
\begin{align*}
\E[R_{t+1}(Q_{t})\vert\mathcal{F}_{t}] & =O,\\
\E[\|R_{t+1}(Q_{t})\|_{2}^{2}\vert\mathcal{F}_{t}] & \le c(\mu_{0})(1+\|Q_{t}\|_{2}^{2}),
\end{align*}
where $c(\mu_{0})>0$ is a constant depending only on $\mu_{0}$ defined
in (\ref{eq:eps}). Applying Lemma~\ref{lem:4.2} with 
\[
\tau_{k}:=\inf\{t:\|Q_{t}\|_2\ge k\},
\]
one has that
\begin{equation}
\lim_{t\to\infty}W_{t}=O\quad\text{on }\Theta\ \text{ a.s.}\label{eq:4.9-3}
\end{equation}

We show (\ref{eq:4.9-2}) by induction. Obviously, $L_{0}^{\varepsilon}=O\le Q_{t}$
for all $t\geq0$. Now we suppose that $L_{k}^{\varepsilon}\le Q_{t}$
holds for all $t\geq t_{k}$ and some random time $t_{k}$, we
shall prove that, there is a random time $t_{k+1}$ such that
$L_{k+1}^{\varepsilon}\le Q_{t}$ for $t\geq t_{k+1}$.

Actually, from $L_{k}^{\varepsilon}\le Q_{t}$ and the monotonicity
of $\F(\cdot)$, we know $\F(L_{k}^{\varepsilon})\leq\F(Q_{t})$.
From Lemma~\ref{lem:det-iter}, there is a random time $t_{k+1}'\geq t_{k}$,
such that 
\begin{equation}
Y_{t}\ge\F(L_{k}^{\varepsilon})-\frac{1}{2}\varepsilon I_{d}\quad\forall t\geq t_{k+1}'.\label{eq:4.9-4}
\end{equation}
According to (\ref{eq:4.9-3}), there is another time $t_{k+1}''$
such that for almost all $\omega\in\Theta$,
\begin{equation}
-\frac{1}{2}\varepsilon I_{d}\leq W_{t}(\omega)\leq\frac{1}{2}\varepsilon I_{d}\quad\forall t\geq t_{k+1}''.\label{eq:4.9-5}
\end{equation}
Combining (\ref{eq:4.9-4}) and (\ref{eq:4.9-5}), and letting $t_{k+1}=\max\{t_{k+1}',t_{k+1}''\}$,
one knows that for almost all $\omega\in\Theta$,
\[
Q_{t}(\omega)=Y_{t}+W_{t}(\omega)\ge\F(L_{k}^{\varepsilon})-\varepsilon I_{d}=L_{k+1}^{\varepsilon}\quad\forall t\geq t_{k+1}.
\]
The proof is complete.
\end{proof}

Now let statement (c) in Theorem~\ref{thm:main} be valid. It follows
from Lemmas~\ref{lem:Leps} and~\ref{prop:comp} that the sequence
$L_{k}^{\eps}$ is increasing, and bounded uniformly with respect
to $k$ and $\eps\in(0,\eps_{0}/2)$, so there are $Q^{\eps}$, uniformly
bounded in $\eps$, such that
\[
\lim_{k\to\infty}L_{k}^{\eps}=Q^{\eps}.
\]
Noticing that $\Phi(\cdot)$ is continuous in the set of all positive
definite $d\times d$ matrices, the above relation along with (\ref{eq:Leps-iter})
implies
\begin{equation}
Q^{\eps}=\F(Q^{\eps})-\eps I_{d}.\label{eq:Q*-eq}
\end{equation}
Moreover, $Q^{\eps}\ge\E[N]-\eps I_{d}\ge\frac{1}{2}\eps_{0}I_{d}$.
So by means of the Bolzano--Weierstrass theorem and the continuity
of $\F(\cdot)$, there is a subsequence of $Q^{\eps}$ converging
to a positive definite matrix, denoted by $Q^{0}$, that satisfies
\[
Q^{0}=\F(Q^{0})=\E[N+\vA^{\tr}\vpi(Q^{0})\vA].
\]
Applying $\vpi(\cdot)$ on both sides, one obtains that $K=\vpi(Q^{0})$
satisfies ARE (\ref{eq:riccati}). This also means that $Q^{0}$ obtained
here, as the fixed point of $\Phi(\cdot)$, is exactly $Q^{*}$ defined
in (\ref{eq:Q*}). 
\begin{rem}
\label{rem:relax}The condition $\E[N]>O$ can be weakened to $\Phi(\E[N])>O$.
Technically, this condition is only used to prove $K>O$ (or equivalently,
$Q^{*}>O$) in (\ref{eq:K-posit}), and to prove Lemma~\ref{prop:comp};
actually, the condition $\Phi(\E[N])>O$ can also ensure these two results. The first one
is straightforward: $Q^{*}=\F(Q^{*})\ge\E[N]$ implies $Q^{*}\ge\Phi(\E[N])>O$
due to the monotonicity of $\F(\cdot)$. To obtain a similar result
as Lemma~\ref{prop:comp}, we define a new lower bound sequence for
each $\eps\in(0,1)$:
\[
L_{0}^{\eps}=(1-\eps)\E[N],\quad L_{k+1}^{\eps}=(1-\eps)\Phi(L_{k}).
\]
One can prove by induction that $L_{k}^{\eps}$ is an increasing sequence
and 
\begin{align*}
\F(L_{k}^{\varepsilon}) & \ge\F(L_{0}^{\varepsilon})=\F((1-\eps)\E[N])\ge(1-\eps)\Phi(\E[N])>O.
\end{align*}
Then one can also obtain the conclusion of Lemma~\ref{prop:comp}
by repeating its proof with two slight changes: i) $L_{0}^{\eps}\le Q_{t}$
a.s. for $t\ge t_{0}$ with some random time $t_{0}$, and ii) (\ref{eq:4.9-4})
and (\ref{eq:4.9-5}) replaced by
\begin{align*}
Y_{t}\ge\F(L_{k}^{\varepsilon})-\frac{1}{2}\varepsilon\F(L_{k}^{\varepsilon}) & \quad\forall t\geq t_{k+1}',\\
-\frac{1}{2}\varepsilon\F(L_{k}^{\varepsilon})\leq W_{t}(\omega)\leq\frac{1}{2}\varepsilon\F(L_{k}^{\varepsilon}) & \quad\forall t\geq t_{k+1}''.
\end{align*}
To see the existence of $t_{0}$, one can introduce a sequence $M_{t}$
with $M_{0}=O$ and $M_{t+1}=(1-\alpha_{t})M_{t}+\alpha_{t}N_{t+1}$.
It is easily seen that $M_{t}\le Q_{t}$ for all $t$ and $M_{t}\to\E[N]$
a.s.; since $\E[N]\ge O$, this implies that there is a time $t_{0}(\omega)$
for almost every $\omega$ such that $M_{t}(\omega)\ge(1-\eps)\E[N]$
for all $t\ge t_{0}(\omega)$, so $Q_{t}(\omega)\ge L_{0}^{\eps}$
a.s. for all $t\ge t_{0}(\omega)$.
\end{rem}

\subsection{From ARE to convergence of $Q_{t}$}

Assume that ARE (\ref{eq:riccati}) has a solution $K$. In Subsection
3.3 we have constructed a convergent sequence $P_{t}$ that dominates
the sequence $\tilde{Q}_{t}=C^{\tr}Q_{t}C$, which means
\[
Q_{t}\le(C^{-1})^{\tr}P_{t}C^{-1}=:\bar{P_{t}}\xrightarrow{\text{a.s.}}(C^{-1})^{\tr}I_{d}C^{-1}=Q^{*},
\]
where $Q^{*}$ is defined in (\ref{eq:Q*}). So the sequence $Q_{t}$
is bounded a.s., and from Lemma~\ref{prop:comp}, there is
a sequence of increasing random times $t_{k}$ such that
\[
L_{k}^{\varepsilon}\le Q_{t}\quad\text{a.s., }\forall t\geq t_{k},
\]
where $L_{k}^{\eps}$ is defined in (\ref{eq:Leps-iter}) with $\eps\in(0,\eps_{0}/2)$.
Now for any $z\in\R^{d}$ and all $k\geq0$, it holds a.s.
that
\begin{gather*}
\limsup_{t\to\infty}z^{\tr}Q_{t}z\leq\lim_{t\to\infty}z^{\tr}\bar{P}_{t}z=z^{\tr}Q^{*}z,\\
\liminf_{t\to\infty}z^{\tr}Q_{t}z\geq\lim_{k\to\infty}z^{\tr}L_{k}^{\varepsilon}z=z^{\tr}Q^{\eps}z,
\end{gather*}
where $Q^{\eps}$ is the limit of $L_{k}^{\eps}$ and satisfies (\ref{eq:Q*-eq}).
Moreover, it has been proved in the last subsection that there is
a subsequence of $Q^{\eps}$ converging to $Q^{*}$. Therefore, 
\[
\lim_{t\to\infty}z^{\tr}Q_{t}z=z^{\tr}Q^{*}z\quad\text{a.s.}
\]
So we can conclude that $Q_{t}$ converges a.s. to a positive definite
matrix $Q^{\star}$ which coincides $Q^{*}=\E[N+\vA^{\tr}K\vA]$. 
\begin{rem}
Although $Q^{\star}$ and $Q^{*}$ are eventually the same, they do
come from different sources: $Q^{\star}$ emerges as the limit of
$Q_{t}$, while $Q^{*}$ is defined via the solution of ARE, or equivalently,
as the fixed point of $\F(\cdot)$.
\end{rem}

To conclude the whole proof of Theorem~\ref{thm:main}, it remains
to verify properties (2)--(4). If $Q_{t}$ converges a.s. to $Q^{\star}$,
then $Q^{\star}=Q^{*}$ is the (unique) fixed point of $\Phi(\cdot)$,
and $K=\Pi(Q^{\star})$ is the unique solution of ARE (\ref{eq:riccati}),
so properties (2) and (4) is proved. In view of Remark~\ref{rem:optimal},
one knows that the optimal control has a  feedback form $u_{t}=\Gamma(Q^{\star})x_{t}$,
which proves property (3).

\section{Proof of Theorem~\ref{thm:stab}}

It has been proved that if LQ problem (\ref{eq:system})--(\ref{eq:cost})
is well-posed, then ARE (\ref{eq:riccati}) has a unique solution
$K$ and the Q-learning process $Q_{t}$ converges to $Q^{\star}$
a.s. By means of the transformation introduced in Subsection 3.3,
we can reformulate the LQ problem into an equivalent form, for which
the solution of ARE and the limit of Q-learning are both identity
matrices. For this reason, we may directly assume, without loss of
generality, that $K=I_{n}$ and $Q^{\star}=I_{d}$. In this case,
the coefficient of optimal feedback control is $\Gamma(I_{d})=O$ (see \eqref{eq:pi} for the definition of  $\Gamma(\cdot)$),
and $\lambda:=\|\E[\vA^{\tr}\vA]\|_{2}<1$. 

Fix a number $\gamma\in(\lambda,1)$. Using the notation (\ref{eq:systems}),
the system under adaptive feedback control $u_{t}^{\mathrm{a}}=\varGamma_{t}x_{t}=\Gamma(Q_{t})x_{t}$
evolves as
\[
x_{t+1}^{\mathrm{a}}=(A_{t+1}+B_{t+1}\varGamma_{t})x_{t}^{\mathrm{a}}=:\bar{A}_{t+1}x_{t}^{\mathrm{a}}.
\]
Since $|x_{t}^{\mathrm{a}}|^{2}$ may not be integrable, we let $\Xi_{t}\subset\Omega$
be an $\mathcal{F}_{t}$-measurable set such that $\E[\mathbf{1}_{\Xi_{t}}|x_{t}^{\mathrm{a}}|^{2}]<\infty$,
and compute
\begin{align*}
\E[\mathbf{1}_{\Xi_{t}}|x_{t+1}^{\mathrm{a}}|^{2}|\mathcal{F}_{t}] & =\E[\mathbf{1}_{\Xi_{t}}|\bar{A}_{t+1}x_{t}^{\mathrm{a}}|^{2}|\mathcal{F}_{t}]\\
 & =\E[\mathbf{1}_{\Xi_{t}}(x_{t}^{\mathrm{a}})^{\tr}\bar{A}_{t+1}^{\tr}\bar{A}_{t+1}x_{t}^{\mathrm{a}}|\mathcal{F}_{t}]\\
 & =(x_{t}^{\mathrm{a}})^{\tr}\E[\mathbf{1}_{\Xi_{t}}\bar{A}_{t+1}^{\tr}\bar{A}_{t+1}|\mathcal{F}_{t}]x_{t}^{\mathrm{a}}\\
 & \leq\mathbf{1}_{\Xi_{t}}\|\E[\mathbf{1}_{\Xi_{t}}\bar{A}_{t+1}^{\tr}\bar{A}_{t+1}|\mathcal{F}_{t}]\|_{2}|x_{t}^{\mathrm{a}}|^{2}.
\end{align*}
Denote $\varPsi_{t}:=\E[\mathbf{1}_{\Xi_{t}}\bar{A}_{t+1}^{\tr}\bar{A}_{t+1}|\mathcal{F}_{t}]$;
as $A_{t+1},B_{t+1}$ are independent of $\mathcal{F}_{t}$, one has
\begin{align*}
\varPsi_{t} & \le\E[A_{t+1}^{\tr}A_{t+1}]+\mathbf{1}_{\Xi_{t}}(\varGamma_{t}^{\tr}\E[B_{t+1}^{\tr}A_{t+1}]+\E[A_{t+1}^{\tr}B_{t+1}]\varGamma_{t}+\varGamma_{t}^{\tr}\E[B_{t+1}^{\tr}B_{t+1}]\varGamma_{t}).
\end{align*}
From the continuity of $\Gamma(\cdot)$ around the identity matrix,
there is a constant $\delta>0$, independent of $t$, such that, as
long as $\|Q_{t}-I_{d}\|_{2}<\delta$, one has that
\[
\|\varGamma_{t}\|_{2}<(\gamma-\lambda)(1+\|\E[B_{t+1}^{\tr}A_{t+1}]\|_{2}+\|\E[A_{t+1}^{\tr}B_{t+1}]\|_{2}+\|\E[B_{t+1}^{\tr}B_{t+1}]\|_{2})^{-1},
\]
and so $\|\varPsi_{t}\|_{2}<\gamma$. Defining, for any $t,s$ with $s>t$,
\[
\Theta_{t}:=\{\omega:\|Q_{t}(\omega)-I_{d}\|_{2}\geq\delta\},\quad\Theta_{t}^{s}:=\cup_{r=t}^{s}\Theta_{r},
\]
we have
\[
\E[\mathbf{1}_{\Xi_{t}\backslash\Theta_{t}}|x_{t+1}^{\mathrm{a}}|^{2}|\mathcal{F}_{t}]\le\gamma\mathbf{1}_{\Xi_{t}\backslash\Theta_{t}}|x_{t}^{\mathrm{a}}|^{2},
\]
and inductively,
\begin{equation}
\begin{alignedat}{1}\E[\mathbf{1}_{\Xi_{t}\backslash\Theta_{t}}|x_{t}^{\mathrm{a}}|^{2}] & \ge\gamma^{-1}\E[\mathbf{1}_{\Xi_{t}\backslash\Theta_{t}}|x_{t+1}^{\mathrm{a}}|^{2}]\\
 & \ge\gamma^{-1}\E[\mathbf{1}_{\Xi_{t}\backslash\Theta_{t}^{t+1}}|x_{t+1}^{\mathrm{a}}|^{2}]\ge\gamma^{-2}\E[\mathbf{1}_{\Xi_{t}\backslash\Theta_{t}^{t+1}}|x_{t+2}^{\mathrm{a}}|^{2}]\\
 & \ge\cdots\ge\gamma^{t-s}\E[\mathbf{1}_{\Xi_{t}\backslash\Theta_{t}^{s-1}}|x_{s}^{\mathrm{a}}|^{2}]\\
 & \ge\gamma^{t-s}\E[\mathbf{1}_{\Xi_{t}\backslash\Theta_{t}^{\infty}}|x_{s}^{\mathrm{a}}|^{2}].
\end{alignedat}
\label{eq:x-a-iter}
\end{equation}

Let $\eps$ be an arbitrary positive number. As $Q_{t}$ converges
to $I_{d}$, there is a random time $\tau$ such that $\Pb(\Theta_{\tau}^{\infty})<\eps/2$.
As $x_{\tau}^{\mathrm{a}}$ is finite, there is a set $\Xi_{\tau}\in\mathcal{F}_{\tau}$
with $\Pb(\Xi_{\tau})>1-\eps/2$ such that $\E[\mathbf{1}_{\Xi_{\tau}}|x_{\tau}^{\mathrm{a}}|^{2}]<\infty$.
From (\ref{eq:x-a-iter}) one has 
\begin{equation}
\E[\mathbf{1}_{\Xi_{\tau}\backslash\Theta_{\tau}^{\infty}}|x_{t}^{\mathrm{a}}|^{2}]\le\E[\mathbf{1}_{\Xi_{\tau}\backslash\Theta_{\tau}^{t}}|x_{t}^{\mathrm{a}}|^{2}]\le\gamma^{t-\tau}\E[\mathbf{1}_{\Xi_{\tau}\backslash\Theta_{\tau}}|x_{\tau}^{\mathrm{a}}|^{2}]\quad\forall\,t>\tau.\label{eq:4-003}
\end{equation}
Set $\Omega_{\eps}=\Xi_{\tau}\backslash\Theta_{\tau}^{\infty}$ and
$c_{\eps}=\E[\mathbf{1}_{\Omega_{\eps}}|x_{\tau}^{\mathrm{a}}|^{2}]<\infty$;
clearly, 
\[
\Pb(\Omega_{\eps})\ge\Pb(\Xi_{\tau})-\Pb(\Theta_{\tau}^{\infty})>1-\eps.
\]
Sum up (\ref{eq:4-003}) with respect to $t$:
\begin{equation}
\E\sum_{t=\tau}^{\infty}\mathbf{1}_{\Xi_{\tau}\backslash\Theta_{\tau}^{t}}|x_{t}^{\mathrm{a}}|^{2}<c_{\eps}\sum_{t=\tau}^{\infty}\gamma^{t-\tau}=\frac{c_{\eps}}{1-\gamma}<\infty.\label{eq:4-001}
\end{equation}
As $\tilde{c}:=\sup_{t\ge\tau}\mathbf{1}_{\Omega_{\eps}}\|\varGamma_{t}\|_{2}$
is bounded, one has
\begin{equation}
\E\sum_{t=\tau}^{\infty}\mathbf{1}_{\Xi_{\tau}\backslash\Theta_{\tau}^{t}}|u_{t}^{\mathrm{a}}|^{2}
=\E\sum_{t=\tau}^{\infty}\mathbf{1}_{\Xi_{\tau}\backslash\Theta_{\tau}^{t}}|\varGamma_{t}x_{t}^{\mathrm{a}}|^{2}
\le\tilde{c}\,\E\sum_{t=\tau}^{\infty}\mathbf{1}_{\Xi_{\tau}\backslash\Theta_{\tau}^{t}}|x_{t}^{\mathrm{a}}|^{2}
<\infty.
\label{eq:4-002}
\end{equation}
Moreover, for $t\ge\tau$ one obtains
\[
\E\bigg\{\mathbf{1}_{\Xi_{\tau}\backslash\Theta_{\tau}^{t}}\begin{bmatrix}x_{t}^{\mathrm{a}}\\
u_{t}^{\mathrm{a}}
\end{bmatrix}^{\tr}\vN_{t+1}\begin{bmatrix}x_{t}^{\mathrm{a}}\\
u_{t}^{\mathrm{a}}
\end{bmatrix}\bigg\}\le\|\E[N]\|_{2}\E[\mathbf{1}_{\Xi_{\tau}\backslash\Theta_{\tau}^{t}}(|x_{t}^{\mathrm{a}}|^{2}+|u_{t}^{\mathrm{a}}|^{2})],
\]
which along with (\ref{eq:4-001}) and (\ref{eq:4-002}) implies that
\[
\E\sum_{t=\tau}^{\infty}\mathbf{1}_{\Xi_{\tau}\backslash\Theta_{\tau}^{t}}\begin{bmatrix}x_{t}^{\mathrm{a}}\\
u_{t}^{\mathrm{a}}
\end{bmatrix}^{\tr}\vN_{t+1}\begin{bmatrix}x_{t}^{\mathrm{a}}\\
u_{t}^{\mathrm{a}}
\end{bmatrix}<\infty.
\]
Noticing that $\Omega_{\eps}\subset\Xi_{\tau}\backslash\Theta_{\tau}^{t}$,
one can obtain from the above argument that
\[
\E\bigg\{\mathbf{1}_{\Omega_{\eps}}\sum_{t=\tau}^{\infty}\bigg(|x_{t}^{\mathrm{a}}|^{2}+|u_{t}^{\mathrm{a}}|^{2}+\begin{bmatrix}x_{t}^{\mathrm{a}}\\
u_{t}^{\mathrm{a}}
\end{bmatrix}^{\tr}\vN_{t+1}\begin{bmatrix}x_{t}^{\mathrm{a}}\\
u_{t}^{\mathrm{a}}
\end{bmatrix}\bigg)\bigg\}<\infty,
\]
and consequently, 
\[
J(x,u_{\cdot}^{\mathrm{a}})+\sum_{t=0}^{\infty}\big(|x_{t}^{\mathrm{a}}|^{2}
+|u_{t}^{\mathrm{a}}|^{2}\big)<\infty\quad\text{on }\ \Omega_{\eps}\ \text{ a.s.}
\]
This concludes the proof of Theorem~\ref{thm:stab} due to the arbitrariness
of $\eps$.

\section{Numerical experiments and discussion}

In this section we illustrate our main results with some examples.

\subsection{Learning rates}

The learning rates $\alpha_{t}$ in Algorithm~\ref{alg:Q-learning}
are superparameters, of which the choice heavily depends on the specific
problem and may affect the speed and accuracy of the algorithm dramatically.

As an example, let us consider LQ problem (\ref{eq:system})--(\ref{eq:cost})
with $n=2$, $m=1$, and
\begin{equation}
\vA_{t}=\vA^{(0)}+w_{t}^{(1)}\vA^{(1)}+w_{t}^{(2)}\vA^{(2)},\quad N_{t}=N,\label{eq:eg-01}
\end{equation}
where $w_{t}^{(1)},w_{t}^{(2)}\sim\mathcal{N}(0,1)$ are independent
random variables, and
\[
\begin{aligned}\vA^{(0)}=\begin{bmatrix}-1 & -0.1 & -0.2\\
2.6 & 0.5 & 0.5
\end{bmatrix} & ,\quad\vA^{(1)}=\begin{bmatrix}0.6 & 0.075 & 0.125\\
-0.8 & 0.1 & -0.375
\end{bmatrix},\\
\vA^{(2)}=\begin{bmatrix}-0.06 & -0.06 & 0.02\\
0.2 & 0.23 & -0.09
\end{bmatrix} & ,\quad N=\begin{bmatrix}3.11 & 1.5626 & -0.2798\\
1.5626 & 1.816175 & -1.021425\\
-0.2798 & -1.021425 & 0.91585
\end{bmatrix}.
\end{aligned}
\]
In this case, the fixed point of the mapping 
\[
\F(Q)=\E[N+\vA_{t}^{\tr}\Pi(Q)\vA_{t}]=N+\sum_{i=0}^{2}(\vA^{(i)})^{\tr}\Pi(Q)\vA^{(i)}
\]
can be solved out explicitly, i.e., $Q^{*}=\F(Q^{*})$ with
\[
Q^{*}=\begin{bmatrix}5 & 2 & 0\\
2 & 2 & -1\\
0 & -1 & 1
\end{bmatrix}.
\]
We apply Algorithm~\ref{alg:Q-learning} to this problem with various
choices of the learning rates:
\[
\alpha_{t}^{(1)}=\frac{1}{t+1},\quad\alpha_{t}^{(2)}=\frac{2}{t+2},\quad\alpha_{t}^{(3)}=\frac{10}{t+10},
\]
and compare the errors $\|Q_{t}-Q^{*}\|_{1}$ within $2000$ time
steps. 
Here we use the 1-norm rather than the 2-norm for reducing the computational cost.

Repeated simulations show that, although the process $Q_{t}$ is convergent
in all three cases, the choice $\alpha_{t}^{(2)}$ has the best overall
performance among the three: the learning process with the lower rate
$\alpha_{t}^{(1)}$ is stable but converges slowly, and with the higher
rate $\alpha_{t}^{(3)}$ it is a bit too fluctuant and unstable; the
moderate rate $\alpha_{t}^{(2)}$ makes a satisfactory balance between
speed and accuracy. Figure~\ref{fig:eg1} gives a sample of the comparative experiment.

\begin{figure}[!tb]
\centering
\includegraphics[width=0.8\textwidth]{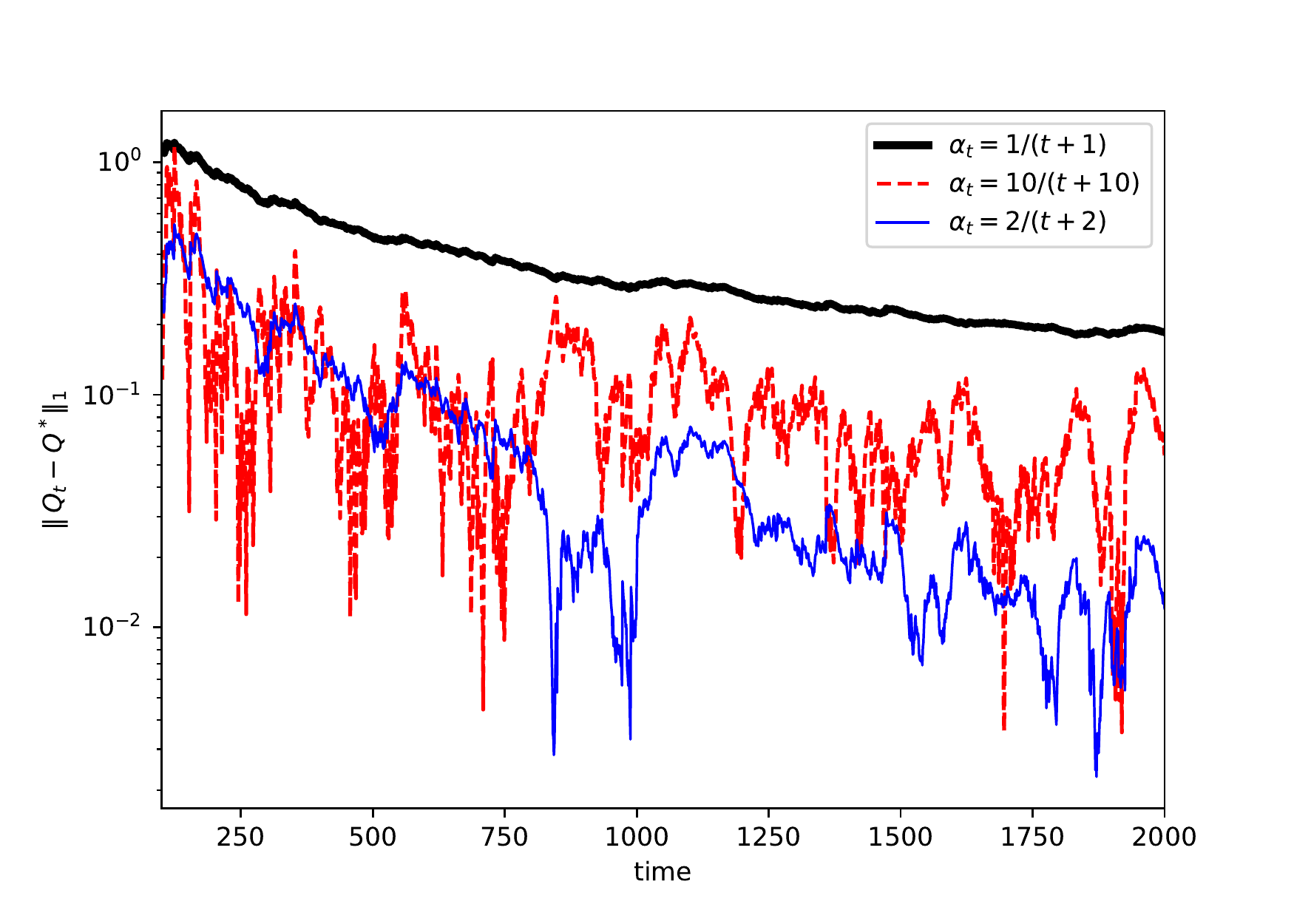}
\caption{Performance comparison with different learning rates.}
\label{fig:eg1}
\end{figure}

\subsection{Discounted problems}

The LQ problem with discounting is very common in applications, in
which the cost function (\ref{eq:cost}) is usually replaced by
\[
J(x,u_{\cdot})=\sum_{t=0}^{\infty}\rho^{2t}\big[x_{t}^{\tr},u_{t}^{\tr}\big]\vN_{t+1}\begin{bmatrix}x_{t}\\
u_{t}
\end{bmatrix}
\]
with a discounted rate $\rho>0$. With a transformation $x_{t}\mapsto\rho^{-t}x_{t}$,
the discounted problem can reduced to our formulation with $\rho\vA_{t}$
instead of $\vA_{t}$, i.e., 
\begin{equation}
x_{t+1}=\rho\vA_{t+1}\begin{bmatrix}x_{t}\\
u_{t}
\end{bmatrix},\quad t=0,1,2,\dots,\label{eq:discount-x}
\end{equation}
subject to the cost function (\ref{eq:cost}).

Evidently, the well-posedness of the discounted problems depends on
the value of the discounted rate $\rho$. Kalman~\cite{kalman1961control}
indicated that there is a critical point $\rho_{\max}>0$ such that
the discounted problems is well-posedness if and only if $\rho<\rho_{\max}$.
He also mentioned that how to determine $\rho_{\max}$ in a general
problem seemed to be very difficult. 

To examine the performance of Algorithm~\ref{alg:Q-learning} around
the critical point, we consider LQ problem (\ref{eq:discount-x})--(\ref{eq:cost})
with $\vA_{t}$ and $N_{t}$ defined in (\ref{eq:eg-01}). A direct
computation gives $\rho_{\max}\approx2.31827$ in this problem. We
apply Algorithm~\ref{alg:Q-learning} for two discounted rates, $\rho_{1}=2.25$
and $\rho_{2}=2.4$, where are very close to $\rho_{\max}$. By means
of Theorem~\ref{thm:main}, the Q-learning process $Q_{t}$ converges
a.s. for $\rho_{1}=2.25$ and diverges a.s. for $\rho_{2}=2.4$. Numerical
simulations have well demonstrated the theoretical result (see Figure~\ref{fig:eg2}). 
The values closer to $\rho_{\max}$ than $\rho_{1}$ and $\rho_{2}$
may also be tested for this purpose, but when they are too near the
critical point, the systematic computation errors would affect the
performance substantially.

\begin{figure}[!tb]
\centering
\includegraphics[width=0.8\textwidth]{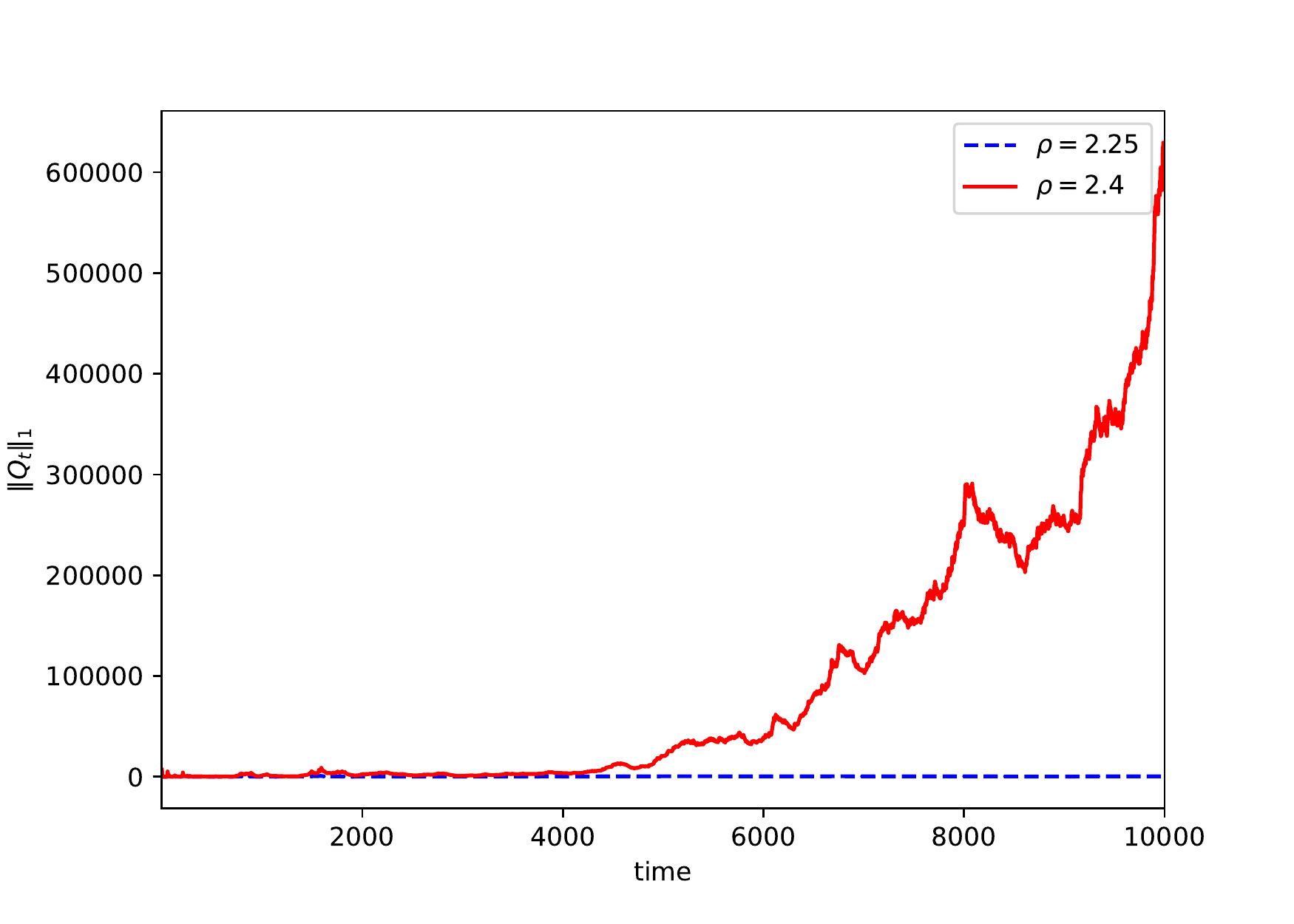}
\caption{Comparison of Q-learning processes for two discounted rates 
around the critical value.}
\label{fig:eg2}
\end{figure}

\subsection{Stabilization}

As far as the stabilization problem is concerned, the cost function
is often set as
\begin{equation}
J(x,u_{\cdot})=\sum_{t=0}^{\infty} |x_{t}|^{2}.\label{eq:cost-stab}
\end{equation}
For discrete-time linear systems with random parameters, this problem
was first discussed by Kalman~\cite{kalman1961control}. Embedded
to our framework, the parameter $N_{t}$ equals $\mathrm{diag}(I_{n},O)$
in this case, which is not positive definite. Nevertheless, in view
of Remark~\ref{rem:relax}, Theorem~\ref{thm:main} can still apply
if $\F(\E[N])$ is positive definite. For the system
\begin{equation}
x_{t+1}=A_{t+1}x_{t}+B_{t+1}u_{t},\label{eq:sys-AB}
\end{equation}
one has that
\[
\F(\E[N])=N+\E\begin{bmatrix}A_{t}^{\tr}\\
B_{t}^{\tr}
\end{bmatrix}\Pi(N)[A_{t},\,B_{t}]=\begin{bmatrix}I_{n}+\E[A_{t}^{\tr}A_{t}] & \E[A_{t}^{\tr}B_{t}]\\
\E[B_{t}^{\tr}A_{t}] & \E[B_{t}^{\tr}B_{t}]
\end{bmatrix}.
\]
To check the condition we may need further information of the parameters.

Let us give a numerical example. Consider LQ problem (\ref{eq:sys-AB})--(\ref{eq:cost-stab})
with $n=2$, $m=1$, and
\begin{align*}
A_{t} & =\me^{w_{t}^{(1)}w_{t}^{(2)}}A^{(1)}-(\sin w_{t}^{(2)})A^{(2)}-\sqrt{w_{t}^{(2)}+w_{t}^{(3)}}A^{(3)},\\
B_{t} & =(w_{t}^{(4)} - w_{t}^{(5)})B_{t}^{(1)}+(\cos w_{t}^{(4)})B_{t}^{(2)},
\end{align*}
where $w_{t}^{(1)},\dots,w_{t}^{(5)}\sim\mathcal{U}(0,1)$ are independent,
and
\begin{align*}
[A^{(1)}, A^{(2)},  A^{(3)},  B^{(1)}, B^{(2)}] & =\rho\begin{bmatrix}-5 & 2 & 0 & -1 & -2 & 3 & -1 & 1\\
2 & 3 & -4 & 7 & 6 & 0 & 1 & 0
\end{bmatrix},
\end{align*}
where $\rho$ is the discounted rate. One can check that $\F(\E[N])>O$
in this example, so Theorems~\ref{thm:main} and (\ref{thm:stab})
can apply to this problem.

We set $\rho=0.25$ with which the problem is demonstrated numerically
to be well-posedness. To verify the stabilization, we conduct two
control policies: the zero control $u_{t}=0$ and the adaptive feedback
control $u_{t}=\Gamma(Q_{t})x_{t}$, and test three initial state
$x_{0}=[1,0]^{\tr}$, $[0,1]^{\tr}$, and $[1,1]^{\tr}$. Numerical
simulations (see Figure~\ref{fig:eg3}) show that the system is stable under
the adaptive feedback control but unstable without control (i.e.,
$u_{t}=0$). 

\begin{figure}[!tb]
\centering
\[
\begin{matrix}
\includegraphics[width=0.314\textwidth]{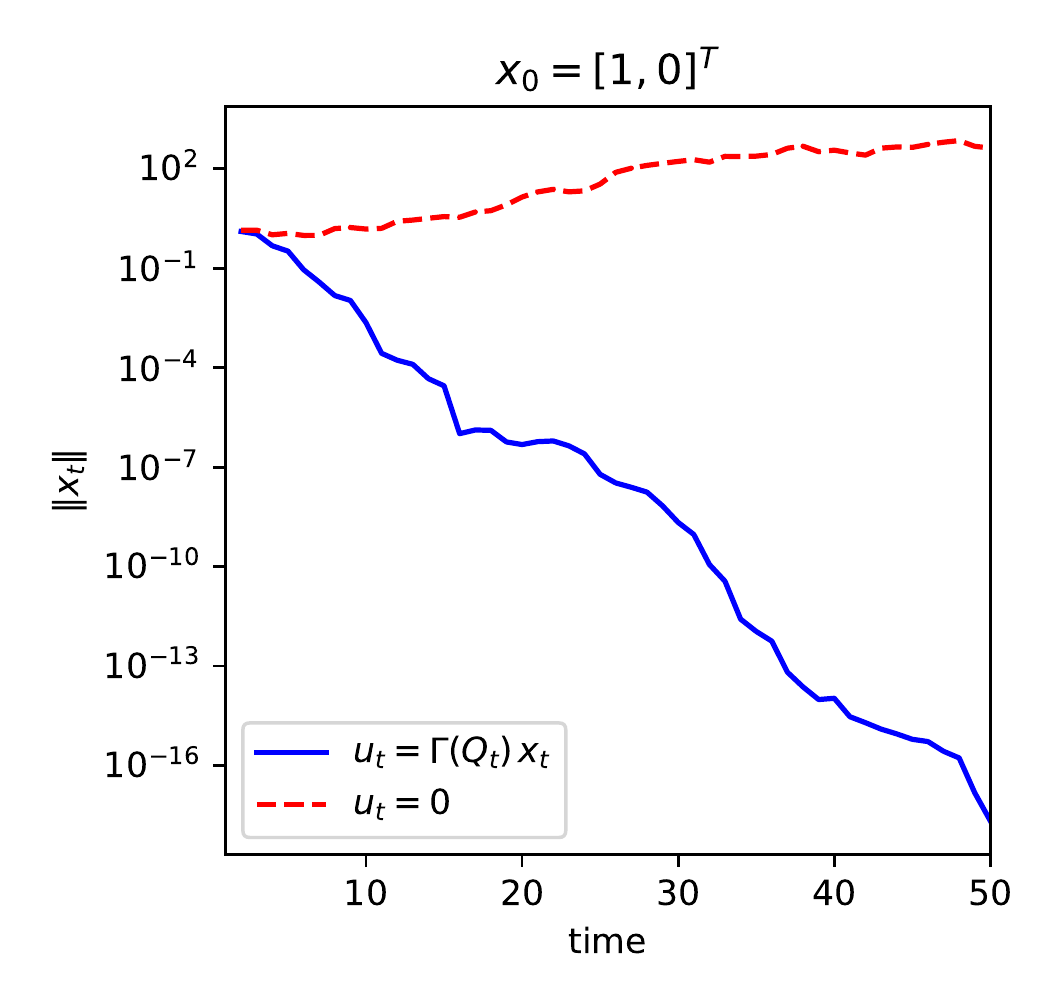} &
\includegraphics[width=0.3\textwidth]{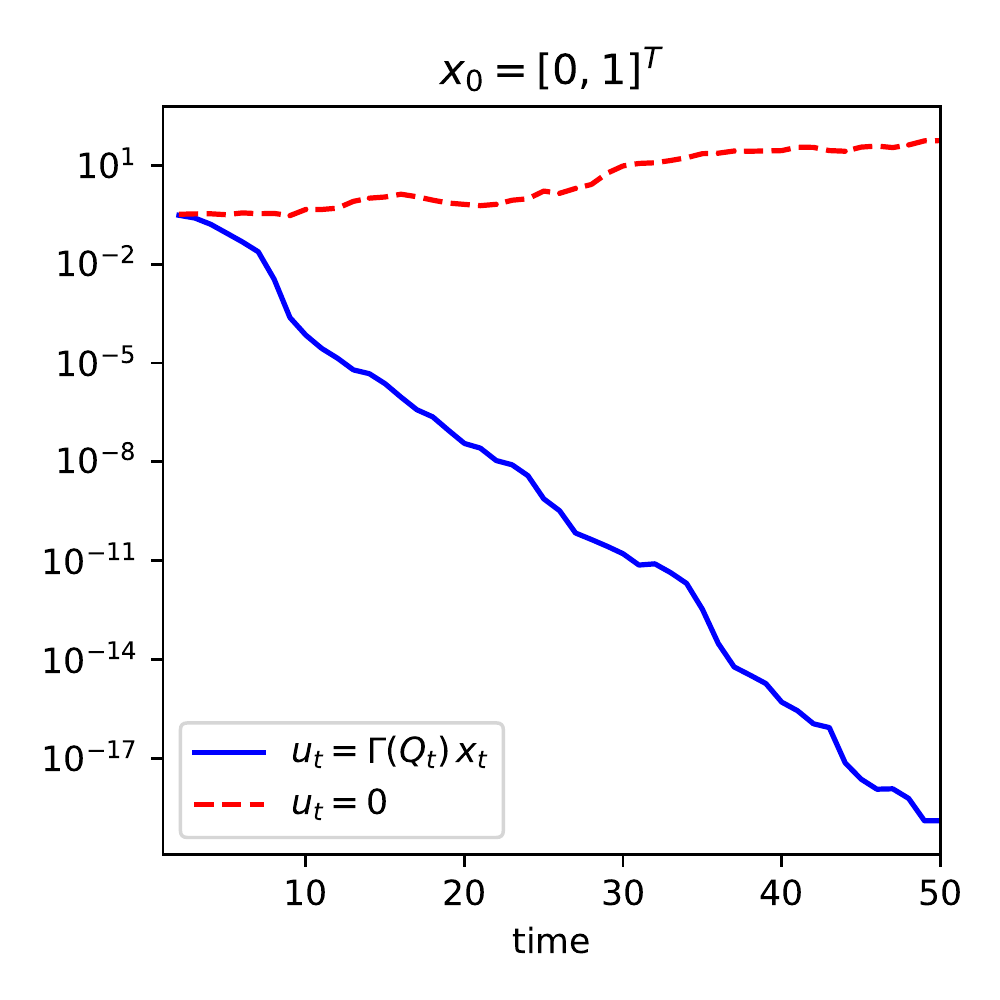} &
\includegraphics[width=0.3\textwidth]{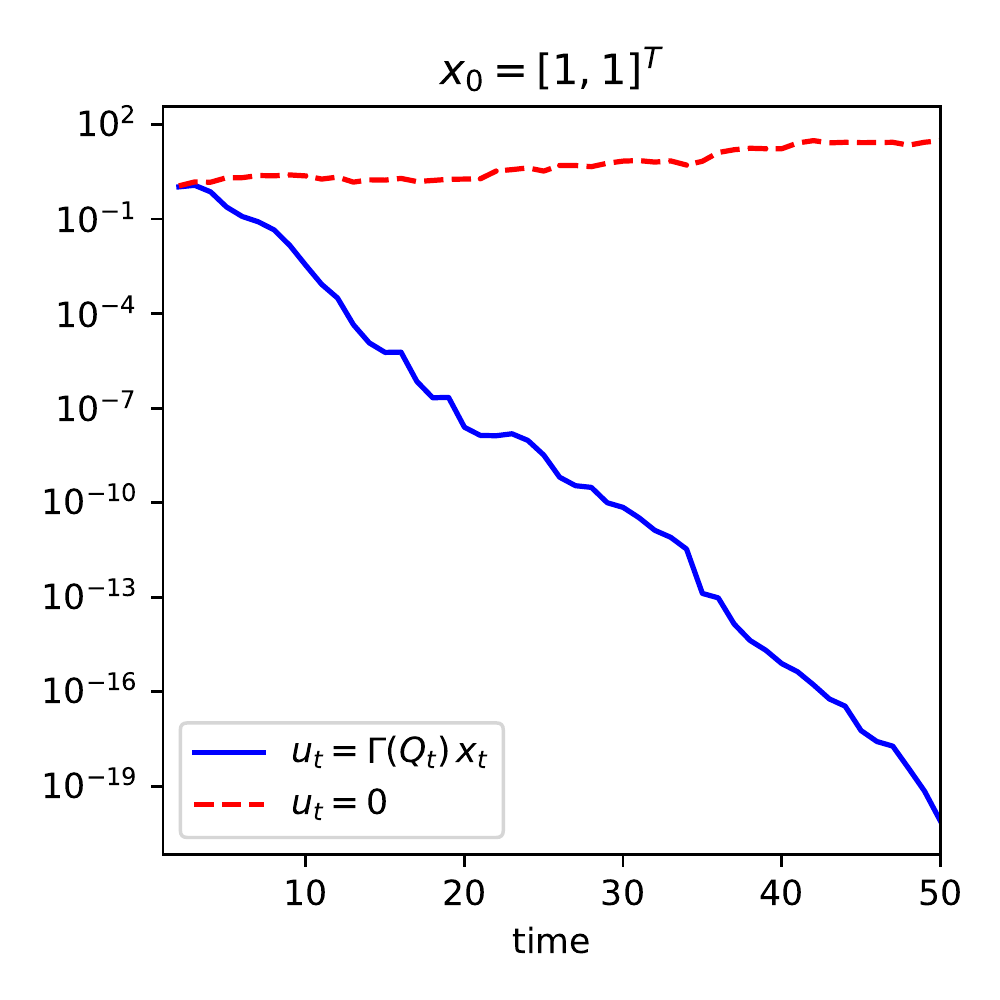}
\end{matrix}
\]
\caption{Comparison of the state processes under adaptive feedback control and zero control.}
\label{fig:eg3}
\end{figure}

\bibliographystyle{siam}
\bibliography{SLQR}

\begin{thebibliography}{10}

\bibitem{alspach1973dual}
{\sc D.~L. Alspach}, {\em A dual control for linear systems with control
  dependent plant and measurement noise}, in 1973 IEEE Conference on Decision
  and Control including the 12th Symposium on Adaptive Processes, IEEE, 1973,
  pp.~681--689.

\bibitem{aoki1967optimal}
{\sc M.~Aoki}, {\em Optimization of Stochastic Systems}, New York: Academic,
  1967.

\bibitem{aoki1976optimal}
\leavevmode\vrule height 2pt depth -1.6pt width 23pt, {\em Optimal Control and
  System Theory in Dynamic Economic Analysis}, vol.~1, North Holland, 1976.

\bibitem{aastrom2013adaptive}
{\sc K.~J. {\AA}str{\"o}m and B.~Wittenmark}, {\em Adaptive control}, Courier
  Corporation, 2013.

\bibitem{Athans1977491}
{\sc M.~Athans, R.~Ku, and S.~B. Gershwin}, {\em The uncertainty threshold
  principle: Some fundamental limitations of optimal decision making under
  dynamic uncertainty}, IEEE Transactions on Automatic Control, 22 (1977),
  pp.~491--495.

\bibitem{Beghi19981031}
{\sc A.~Beghi and D.~D'Alessandro}, {\em Discrete-time optimal control with
  control-dependent noise and generalized riccati difference equations},
  Automatica, 34 (1998), pp.~1031--1034.

\bibitem{bertsekas2019reinforcement}
{\sc D.~P. Bertsekas}, {\em Reinforcement learning and optimal control}, Athena
  Scientific Belmont, MA, 2019.

\bibitem{chen1986optimal}
{\sc H.-F. Chen and L.~Guo}, {\em Optimal stochastic adaptive control with
  quadratic index}, International Journal of Control, 43 (1986), pp.~869--881.

\bibitem{chen1998stochastic}
{\sc S.-P. Chen, X.-J. Li, and X.-Y. Zhou}, {\em Stochastic linear quadratic
  regulators with indefinite control weight costs}, SIAM Journal on Control and
  Optimization, 36 (1998), pp.~1685--1702.

\bibitem{chow1975analysis}
{\sc G.~C. Chow}, {\em Analysis and Control of Dynamic Economic Systems},
  Wiley, 1975.

\bibitem{DeKoning1982443}
{\sc W.~L. De~Koning}, {\em Infinite horizon optimal control of linear discrete
  time systems with stochastic parameters}, Automatica, 18 (1982),
  pp.~443--453.

\bibitem{drenick1964optimal}
{\sc R.~Drenick and L.~Shaw}, {\em Optimal control of linear plants with random
  parameters}, IEEE Transactions on Automatic Control, 9 (1964), pp.~236--244.

\bibitem{duncan1999adaptive}
{\sc T.~E. Duncan, L.~Guo, and B.~Pasik-Duncan}, {\em Adaptive continuous-time
  linear quadratic gaussian control}, IEEE Transactions on Automatic Control,
  44 (1999), pp.~1653--1662.

\bibitem{Dvoretzky1956}
{\sc A.~Dvoretzky}, {\em On stochastic approximation}, in Proceedings of the
  Third Berkeley Symposium on Mathematical Statistics and Probability, vol.~1,
  University of California Press, 1956, pp.~39--56.

\bibitem{faradonbeh2020adaptive}
{\sc M.~K.~S. Faradonbeh, A.~Tewari, and G.~Michailidis}, {\em On adaptive
  linear--quadratic regulators}, Automatica, 117 (2020), p.~108982.

\bibitem{Harris1978213}
{\sc S.~E. Harris}, {\em Stochastic controllability of linear discrete systems
  with multiplicative noise}, International Journal of Control, 27 (1978),
  pp.~213--227.

\bibitem{huang2006infinite}
{\sc Y.~Huang, W.-H. Zhang, and H.-S. Zhang}, {\em Infinite horizon {LQ}
  optimal control for discrete-time stochastic systems}, in 2006 6th World
  Congress on Intelligent Control and Automation, vol.~1, IEEE, 2006,
  pp.~252--256.

\bibitem{jaakkola1994convergence}
{\sc T.~Jaakkola, M.~I. Jordan, and S.~P. Singh}, {\em Convergence of
  stochastic iterative dynamic programming algorithms}, in Advances in Neural
  Information Processing Systems, 1994, pp.~703--710.

\bibitem{kalman1961control}
{\sc R.~E. Kalman}, {\em Control of randomly varying linear dynamical systems},
  in Proceedings of Symposia in Applied Mathematics, 1961, pp.~287--298.

\bibitem{Kalman1959405}
{\sc R.~E. Kalman and J.~E. Bertram}, {\em A unified approach to the theory of
  sampling systems}, Journal of the Franklin Institute, 267 (1959),
  pp.~405--436.

\bibitem{Ku1977866}
{\sc R.~T. Ku and M.~Athans}, {\em Further results on the uncertainty threshold
  principle}, IEEE Transactions on Automatic Control, 22 (1977), pp.~866--868.

\bibitem{martin1975stability}
{\sc D.~N. Martin and T.~L. Johnson}, {\em Stability criteria for discrete-time
  systems with colored multiplicative noise}, in 1975 IEEE Conference on
  Decision and Control including the 14th Symposium on Adaptive Processes,
  IEEE, 1975, pp.~169--175.

\bibitem{Morozan198389}
{\sc T.~Morozan}, {\em Stabilization of some stochastic discrete-time control
  systems}, Stochastic Analysis and Applications, 1 (1983), pp.~89--116.

\bibitem{ni2015indefinite}
{\sc Y.-H. Ni, X.~Li, and J.-F. Zhang}, {\em Indefinite mean-field stochastic
  linear-quadratic optimal control: from finite horizon to infinite horizon},
  IEEE Transactions on Automatic Control, 61 (2015), pp.~3269--3284.

\bibitem{Pronzato1996855}
{\sc L.~Pronzato, C.~Kulcs\'ar, and E.~Walter}, {\em An actively adaptive
  control policy for linear models}, IEEE Transactions on Automatic Control, 41
  (1996), pp.~855--858.

\bibitem{rami2002indefinite}
{\sc M.~A. Rami, J.~B. Moore, and X.~Y. Zhou}, {\em Indefinite stochastic
  linear quadratic control and generalized differential {R}iccati equation},
  SIAM Journal on Control and Optimization, 40 (2002), pp.~1296--1311.

\bibitem{ran1993linear}
{\sc A.~C.~M. Ran and H.~L. Trentelman}, {\em Linear quadratic problems with
  indefinite cost for discrete time systems}, SIAM Journal on Matrix Analysis
  and Applications, 14 (1993), pp.~776--797.

\bibitem{sutton2018reinforcement}
{\sc R.~S. Sutton and A.~G. Barto}, {\em Reinforcement Learning: An
  Introduction}, MIT press, 2018.

\bibitem{Tiedemann1984449}
{\sc A.~R. Tiedemann and W.~L. De~Koning}, {\em The equivalent discrete-time
  optimal control problem for continuous-time systems with stochastic
  parameters}, International Journal of Control, 40 (1984), pp.~449--466.

\bibitem{tse1973actively}
{\sc E.~Tse and Y.~Bar-Shalom}, {\em An actively adaptive control for linear
  systems with random parameters via the dual control approach}, IEEE
  Transactions on Automatic Control, 18 (1973), pp.~109--117.

\bibitem{tsitsiklis1994asynchronous}
{\sc J.~N. Tsitsiklis}, {\em Asynchronous stochastic approximation and
  {Q}-learning}, Machine Learning, 16 (1994), pp.~185--202.

\bibitem{Wang2016379}
{\sc T.~Wang, H.-G. Zhang, and Y.~Luo}, {\em Infinite-time stochastic linear
  quadratic optimal control for unknown discrete-time systems using adaptive
  dynamic programming approach}, Neurocomputing, 171 (2016), pp.~379--386.

\bibitem{watkins1989learning}
{\sc C.~J.~C.~H. Watkins}, {\em Learning from delayed rewards}, Ph.D. Thesis,
  University of Cambridge,  (1989).

\bibitem{yaz1985stabilization}
{\sc E.~Yaz}, {\em Stabilization of deterministic and stochastic-parameter
  discrete systems}, International Journal of Control, 42 (1985), pp.~33--41.

\bibitem{Yaz1988407}
\leavevmode\vrule height 2pt depth -1.6pt width 23pt, {\em Control of randomly
  varying systems with prescribed degree of stability}, IEEE Transactions on
  Automatic Control, 33 (1988), pp.~407--410.

\end{thebibliography}

\end{document}